\documentclass[11pt]{article}
\usepackage{amsmath}
\usepackage{amsfonts} \makeatother
\usepackage{amsfonts,amsmath,amssymb,mathrsfs}
\usepackage{cite}
\usepackage{enumitem}
\usepackage{tikz-cd}
\usepackage{amsthm}
\usepackage[colorlinks=true,urlcolor=blue,
citecolor=red,linkcolor=blue,linktocpage,pdfpagelabels,
bookmarksnumbered,bookmarksopen]{hyperref}

\numberwithin{equation}{section}
\newtheorem{theorem}{Theorem}[section]
\newtheorem{lemma}[theorem]{Lemma}

\newtheorem{remark}{Remark}[section]

\newtheorem{corollary}[theorem]{Corollary}

\newcommand{\be}{\begin{equation}}
	\newcommand{\ee}{\end{equation}}
\newcommand\bes{\begin{eqnarray}}
	\newcommand\ees{\end{eqnarray}}
\newcommand{\bess}{\begin{eqnarray*}}
	\newcommand{\eess}{\end{eqnarray*}}

\newcommand\inv{\frak i}

\DeclareMathOperator*{\einf}{ess\,inf}

\DeclareMathOperator{\sgn}{sgn}
\DeclareMathOperator{\cK}{\mathcal K}
\DeclareMathOperator{\cF}{\mathcal F}

\newcommand{\id}{\textnormal{id}}

\newcommand\R{\mathbb{R}}

\usepackage{fancyhdr}
\fancypagestyle{mypagestyle}{%
	\fancyhf{}
	\fancyhead[OC]{O. Cabrera Chávez}
	\fancyhead[EC]{Chain-structure solutions to a Schrödinger-Poisson system in $\R^3$}
	\fancyfoot[C]{\thepage}%
}
\author{Omar Cabrera Chavez\footnote{O. Cabrera Chavez was fully supported by DAAD Research Grants - Doctoral Programmes in Germany, 2020/21 (57507871)}}
\title{Chain-structure solutions to a Schrödinger-Poisson system in $\R^3$}
\date{\today}

\begin{document}
	\maketitle
\begin{abstract}
	We prove the existence of ground states and high-energy solutions to the following Schrödinger-Poisson system
	\begin{align*}
		\begin{cases}
			- \Delta u + a(x) u + u v = 0 \\
			\Delta v = u^2
		\end{cases} 
		\quad \text{in } \R^3,
	\end{align*}
	where $a \in L^\infty(\R^3)$ is nonnegative and radially symmetric in the first two variables. Differing from the standard approach, our framework yields \textit{chain-structure solutions}, i.e. solutions periodic in the third variable. A central part of this work is the construction of the Green function of a Poisson problem subject to periodic boundary conditions and we show that its asymptotic profile is tightly related to both the two and three dimensional Poisson problems in the entire space. If the potential $a$ is constant along the third variable, we apply symmetry techniques to construct solutions that have nonvanishing derivative in the third variable.

	\noindent\textbf{Keywords:} Schrödinger-Poisson system, Elliptic partial differential equation, logarithmic Choquard equation, Green function, ground states
	
	\noindent\textbf{Mathematics Subject Classification:} 35J47, 35J20, 35A15, 35J08.
\end{abstract}
	
\section{Introduction}
We consider the following \textit{Schrödinger-Poisson system}
\begin{align}\label{eq:SrPssystem}
\begin{cases}
	- \Delta u + a(x) u + u v = 0 ,\\
	\Delta v = u^2.
\end{cases} 
\text{in } \R^N,
\end{align}
where $a \in L^\infty(\R^N)$. Using the fundamental solution of the Laplace operator $-\Delta$, we may invert the second equation in \eqref{eq:SrPssystem} up to harmonic functions, so we may set
\begin{equation}\label{eq:fundamental_laplace}
	v = (\Psi_N \ast u^2)(x) = \int_{\R^N} \Psi_N(x-y) u^2(y) dy \qquad
	\Psi_N(x) := 
	\begin{aligned}
		\begin{cases}
			\frac{-1}{N(2-N)\omega_N} |x|^{2-N} \quad &\text{for } N \geq 3,\\
			\frac{1}{2\pi} \log|x| \quad &\text{for } N = 2,
		\end{cases}
	\end{aligned}
\end{equation}
which yields the \text{Choquard equation}
\begin{equation}\label{eq:choquard_general}
	- \Delta u + a(x) u + (\Psi_N \ast u^2)u = 0 \quad \text{in } \R^N.
\end{equation}

The equation with $N = 3$ and for $a(x) \equiv a >0$ was introduced by Pekar \cite{p} in 1954 to describe the behavior of polaron at rest, and we find this equation in a number of other applications in Physics, e.g. \cite{lieb,penrose}. We will focus on the variational aspect of this equation in the following.

Our aim is to find \textit{chain-structure} solutions to the system \eqref{eq:SrPssystem} in $\R^3$, i.e. solutions which are periodic in the third variable. More precisely, let $\ell > 0$ and define $\Omega := \R^2 \times (-\ell, \ell)$, then consider the following system subject to periodic boundary conditions.
\begin{align}\label{eq:SP_system_periodic}
	\begin{cases}
		- \Delta u + a(x) u + u v = 0  \quad &\text{in } \Omega,\\
		\Delta v = u^2 \quad &\text{in } \Omega,\\
		u(\cdot,-\ell) = u(\cdot,\ell), \quad \partial_3 u(\cdot,-\ell) = \partial_3 u(\cdot,\ell) \quad &\text{in } \R^2,
	\end{cases} 
\end{align}
where $a \in L^\infty(\Omega)$ and $\einf_{x \in \Omega}a(x) > 0$.

The equation \eqref{eq:choquard_general} has been studied extensively under variant conditions, so we refer the reader to the survey \cite{mvs} and the references therein for detailed overview on the subject. Our focus will be to contrast the differences in the structure of \eqref{eq:SrPssystem} in the logarithmic potential case $N=2$ and the inverse power case $N =3$, and how the periodic system \eqref{eq:SP_system_periodic} shares some qualitative properties with both of them.

The literature for the planar case $N = 2$ is comparatively smaller that the higher dimensional case $N \geq 3$ due to the logarithmic kernel in \eqref{eq:choquard_general} which changes sign and is unbounded. This is a feature that will also be present in this work. In \cite{stubbe} Stubbe works with \eqref{eq:choquard_general} with a constant potential $a(x) \equiv a \in \R$. He overcomes the difficulties involving the logarithmic internal potential by working in a smaller space Hilbert space 
$$X := \{ u \in H^1(\R^2) : \log(1 + | \cdot |) u^2 \in L^1(\R^2) \}.$$
In this space the logarithmic convolution is well-defined. Within this framework he shows that there is a unique positive and radially symmetric ground state for $a > 0$ using symmetric rearrangement inequalities and describes a bifurcation phenomenon for the ground states for negative but bounded values of $a$. Inspired by this variational framework, other results have been produced under different assumptions, e.g. \cite{c,cw,dw,gw,ccw}.

Now, for the higher dimensional case $N \geq 3$, we have the work of Lieb \cite{lieb}, where he proved the existence of a unique ground state to \eqref{eq:choquard_general} using minimization arguments. Lions then showed in \cite{lions} the existence of infinitely many distinct spherically symmetric solutions when $a(x) \geq 0$ and radially symmetric. 
We refer the interested reader to \cite{ak,ccs,fl,mz,mvs2} for more results under different hypotheses.

In these cases, the Newton, Riesz or related positive potentials appear \eqref{eq:fundamental_laplace} in the variational framework but recent works have analyzed the higher dimensional equation with more general unbounded sign-changing potentials. In \cite{gw} we see the equation \eqref{eq:choquard_general} in the case $N \geq 3$, where the fundamental solution $\Psi_N$ is substituted for $\log|\cdot|$. Here the existence of a ground state and a mountain-pass solution is proved. In \cite{bvs} the authors work with general unbounded potentials by working on a relaxed problem arising from a truncation of the positive part of the potential and then passing to the limit in the truncation parameter.

We would like to approach our problem using a similar strategy as the one described above for the system in the full space. Then we need to understand the corresponding problem for the fundamental solution of the periodic Poisson equation
\begin{align}\label{eq:poisson_per}
	\begin{cases}
		\Delta w(x) = \delta(x-y) \quad &x \in \Omega,\\
			w(\cdot,-\ell) = w(\cdot,\ell) \quad &\text{in } \R^2, \\
			\partial_3 w(\cdot,-\ell) = \partial_3 w(\cdot,\ell) \quad &\text{in } \R^2,
	\end{cases}
\end{align}
where $y \in \Omega$. 
Our first important result regards the existence and an asymptotic description of the Green function of \eqref{eq:poisson_per}. Before the statement we introduce some notation.

We will say that two functions $f,g: \R^N \to \R$ are \textit{asymtotically equivalent at infinity}, if
$$ 0 < \lim_{|x| \to \infty} \frac{f(x)}{g(x)} < \infty.$$
For short, we write this as
$$ f(x) \sim g(x) \qquad \text{as } |x| \to \infty.$$
Now, for any function $f : \Omega \to \R$, we denote the periodic extension of $f$ onto $\R^3$ by $\overline{f} : \R^3 \to \R$.
\begin{theorem}\label{thm:asymtoticsGreen}
  There is a function $K$ of class $C^\infty(\Omega_*^2)$, where $\Omega_*^2 =\overline{\Omega} \times \overline{\Omega} \setminus \{ (x,y) \in \Omega :  x = y\}$, with the following properties
  \begin{enumerate}
  \item $K = K_1+ K_2$, where $K_1(x,y) = -\frac{1}{4 \pi |x-y|}$ and $K_2 \in C^{\infty}(\overline{\Omega} \times \overline{\Omega})$ is symmetric in $x$ and $y$, and satisfies
    $$
    K_2(x,y) \sim  \log(1+|x-y|) \quad  \text{as } |x-y| \to \infty.
    $$
    Moreover, the periodic extension $\overline{K} : \R^3 \to \R$ satisfies $\overline{K} \in C^\infty(\R^6_*)$, where $\R^6_* := (\R^3 \times \R^3 )\setminus \{(x,y) \in \R^3 : x = y\}$.

  \item For $f \in C^\infty_c(\Omega)$, the function
    $$
    \cK f : \overline{\Omega} \to \R, \qquad [\cK f](x)= \int_{\Omega} K(x,y)f(y) dy
    $$
    defines a classical solution of
    	\begin{equation*}
    	\begin{cases}
    		\begin{alignedat}{3}
    			\Delta w(x)            &= f(x)             &\quad  &\text{in } \Omega,\\
    			w(\cdot,-\ell)            &= w(\cdot,\ell)             &\quad  &\text{in } \R^2,\\
    			\partial_{3}w(\cdot,-\ell)&= \partial_{3}w(\cdot,\ell) &\quad  &\text{in } \R^2,
    		\end{alignedat}
    	\end{cases}
    \end{equation*}
    and the periodic extension satisfies $\overline{\cK f} \in C^\infty(\R^3)$.

\end{enumerate}
  \end{theorem}
Thus, we see that the Green function $K$ for the periodic Poisson problem behaves asymptotically like the fundamental solution of the Poisson equation in the $2$-dimensional case at infinity and like its $3$-dimensional counterpart at $0$. Moreover, to contrast the asymptotic logarithmic profile of $K$ at infinity, we note that the Green function corresponding to the Dirichlet version of \eqref{eq:poisson_per} decays to $0$ at infinity as a consequence of the maximum principle.

With this description of the Green function at hand, we prove the following result
\begin{theorem}\label{thm:wsolution}
	Let $u \in H^1(\Omega)$ satisfy
	$$ \int_\Omega \log(1+|x|)u^2(x) dx < \infty, $$
	then the function
	$$ w(x): = \cK [u^2](x) = \int_{\Omega} K(x,y)u^2(y) dy, $$
	is a distributional solution (see \eqref{eq:dist_solution} below) to 
	$$  \Delta w(x) = u^2(x) \qquad \text{in } \Omega,$$
	belonging to $W^{3,3/2}_{loc}(\Omega) \cap C_{loc}^{0,\alpha}(\Omega)$ for all $\alpha \in (0,1)$. Moreover it satisfies
	 $$ w(x) \sim \log(1+|x|) \quad \text{ as } |x| \to \infty, \, x \in \Omega.$$
	and the periodic extension $\overline{w} : \R^3 \to \R$ belongs to the space $C^{0,\alpha}_{loc}(\R^3)$ for all $\alpha \in (0,1)$.
\end{theorem}
Now, we may continue our initial idea and plug 
$$w(x) = \cK [u^2](x) = \int_\Omega K(x,y)u^2(y) dy,$$ 
back in \eqref{eq:SP_system_periodic} which yields the following Choquard-type equation subject to periodic boundary conditions
\begin{equation}\label{eq:periodic_choquard}
\begin{cases}
		- \Delta u + a(x) u + \cK [u^2]u = 0 \quad &x \in \Omega,\\
	u(\cdot,-\ell) = u(\cdot,\ell), \quad \partial_3 u(\cdot,-\ell) = \partial_3 u(\cdot,\ell) \quad &\text{in } \R^2.
\end{cases}
\end{equation}
Then, at least formally, we associate to this equation the energy functional
\begin{equation}\label{eq:energy_functional}
	\Phi(u) := \frac{1}{2} \int_{\Omega} |\nabla u|^2 + a(x) u^2 dx +  \frac{1}{4} \int_{\Omega} \int_{\Omega} K(x,y) u^2(x) u^2(y) dx dy. 
\end{equation}
From the asymptotic description $\cK [u^2](x) \sim \log(1 +|x|)$ as $|x| \to \infty$ we can see that the double integral in \eqref{eq:energy_functional} is in general not well defined in the usual Sobolev space $H^1(\Omega)$, as was the case in the planar Choquard equation. Then, any variational framework has to accommodate to the sign-changing internal potential $K$ which is unbounded on both sides as well. To overcome this, we follow Stubbe's approach, see \cite{stubbe}, and carry out our analysis in the subspace\footnote{Here $H^1_p(\Omega)$ represents the Sobolev space subject to periodic boundary conditions $u(\cdot,-\ell) = u(\cdot,\ell)$ in the sense of traces, see \eqref{eq:Sobolev_periodic} below.}
$$ X := \left\{ u \in H^1_p(\Omega) : \int_{\Omega} \log(1+|x|) u^2(x) dx < \infty \right\}. $$
Proceeding in this way, we will see that $\Phi$ is well-defined in $X$ and, as expected, critical points of $\Phi$ correspond to weak solutions to \eqref{eq:periodic_choquard}. Within this framework we will obtain the existence of ground states and high-energy solutions based on Ljusternik-Schnirelmann theory applied to the restriction of $\Phi$ to the Nehari manifold
$$ \mathcal{N} := \{ u \in X \setminus \{0\} : \Phi'(u)u = 0\}.$$

For our existence results we focus on the radially symmetric case. We introduce the following subspace
$$ X_r := \left\{u \in X : u(Ax',x_3) = u(x',x_3) \quad \text{ a.e. in } \Omega \text{ for all } A \in O(2)\right\}. $$
Its elements are radially symmetric functions in the $x'$-variable. Further, we denote $\mathcal{N}_r := \mathcal{N} \cap X_r$, which contains all critical points that are radially symmetric in the $x'$-variable.

Now, we introduce the notion of \textit{radial ground state}, i.e. a function $u_r \in X_r$ satisfying $\Phi'(u_r) = 0$ and 
$$ \Phi(u_r) = \inf \{ \Phi(u) : u \in X_r \setminus \{0\}, \, \Phi'(u) = 0\} =: c_r $$

Our first result regards the existence of radial ground states and high energy solutions to \eqref{eq:periodic_choquard}
\begin{theorem}\label{thm:high_energy_radial}
	Let $a \in L^{\infty}(\Omega)$ satisfy $\einf_{\Omega}a > 0$ and
	$$	a(Ax',x_3) = a(x', x_3) \quad \text{ a.e. in } \Omega \text{ for all } A \in O(2).$$
	Then \eqref{eq:periodic_choquard} admits a sequence of solution pairs $\{\pm u_n\}_n \subset X_r$ such that $\Phi(\pm u_n) \to \infty$ as $n \to \infty$ and $\pm u_1$ satisfies
	$$ \Phi(\pm u_1) = c_r.$$ 
	Moreover, if the periodic extension satisfies $\overline{a} \in C^{0,\alpha}_{loc}(\R^3)$, then the periodic extensions $\{\pm \overline{u}_n\}_n$ are of class $C^{2,\alpha}_{loc}(\R^3)$.
\end{theorem}

Now we consider the particular case $ \partial_3 a \equiv 0$. Here, we write $a(x',x_3) = a(x')$ for simplicity. In this case, we have the solutions to \eqref{eq:SPsystem} arising from the planar Choquard equation
$$ -\Delta u + a(x') u + (\log|\cdot| \ast u^2) u = 0 \qquad x' \in \R^2, $$
by extending as a constant function of $x_3$. Conversely, a solution to \eqref{eq:periodic_choquard} with vanishing $x_3$-derivative yields a solution to the planar equation by fixing any value $x_3 \in [-\ell, \ell]$, see Section \ref{constant_potentials} below. Therefore we are interested in finding solutions for \eqref{eq:periodic_choquard} such that $\partial_3 u \not \equiv 0$ a.e. in $\Omega$. We will call such a solution a \textit{fully nontrivial}. Then, we have the following
\begin{theorem}\label{thm:fully_nontrivial}
	Let $\partial_3 a \equiv 0$, $\einf_{\Omega} a > 0$ and
	$$ a(Ax',x_3) = a(x',x_3) \quad \text{a.e. in } \Omega \text{ for all } A \in O(2). $$
	Then, there is $\ell_* > 0$ such that, for all $\ell > \ell_*$, if $u_g \in X \setminus \{0\}$ satisfies
	$$ \Phi'(u_g) = 0 \quad \text{and} \quad \Phi(u_g) = c_r,$$
	then $u_g$ is fully nontrivial.
\end{theorem}
In this result, the minimizing property of the radial ground state is key to show that $u_g$ is fully nontrivial and solutions at higher energy levels might not have this property anymore.

To overcome this issue and produce fully nontrivial solutions for all $\ell > 0$ we introduce further symmetries to our setting. We consider the following isometric involution in the space $X$
$$ \inv(u)(x) := -\overline{u}(x',x_3-\ell) \quad x \in \Omega, $$
where $\overline{u}:\R^3 \to \R$ denotes the periodic extension of $u$.

Now, we introduce the following subspace of functions invariant under the action of $G := O(2) \times \{id,\inv\}$
$$ X_G := \{ u \in X : \inv(u) = u  \text{ and } u(Ax',x_3) = u(x) \text{ a.e. in } \Omega \text{ for all } A \in O(2)\}. $$
We have the corresponding notion of \textit{$G$-invariant ground state}, i.e. a function $u_g \in X_g$ such that $\Phi'(u_g) = 0$ and
$$ \Phi(u_g) = \inf \{ \Phi(u) : u \in X_G \setminus \{0\}, \, \Phi'(u) = 0\} =: c_G.$$
The advantage of working in this setting, is that a nontrivial critical point $u \in X_G$ cannot be constant along the $x_3$-variable, since $\overline{u}(x',x_3 - \ell) = -u(x',x_3)$ for all $x \in \Omega$. Thereby this issue is taken care of by the functional setting. Proceeding in this way, we obtain the following

\begin{theorem}\label{thm:high_energy_nontrivial}
Let $\partial_3 a \equiv 0$, $\einf_{\Omega} a > 0$ and
$$ a(Ax',x_3) = a(x',x_3) \quad \text{a.e. in } \Omega \text{ for all } A \in O(2). $$
Then \eqref{eq:periodic_choquard} admits a sequence of fully nontrivial solution pairs $\{\pm u_n\}_n \subset X_G$ such that $\Phi(\pm u_n) \to \infty$ as $n \to \infty$ and $\pm u_1$ satisfies
	$$ \Phi(\pm u_1) = c_G.$$ 
Moreover, if the periodic extension satisfies $\overline{a} \in C^{0,\alpha}_{loc}(\R^3)$, then the periodic extensions $\{\pm \overline{u}_n\}_n$ are of class $C^{2,\alpha}_{loc}(\R^3)$.
\end{theorem}

This paper is organized as follows. In Section \ref{section:green_function} we carry out a detailed analysis of Green function corresponding to the problem Poisson problem \eqref{eq:poisson_per}. We then dedicate Section \ref{sec:poisson} to the formulation and analysis of the boundary value problem \eqref{eq:poisson_per}. In Section \ref{sec:variational_setting} we present our variational setting and investigate the structure of the Nehari manifold. Finally in Sections \ref{multiplicity} and \ref{constant_potentials} we will present our existence results for the general radially symmetric case and in the particular case for constant potentials.

Throughout this work we make use of generic labels $C,C_1,C_2$, etc. to name constants whose exact value is not relevant to the argument. These may change from line to line when there is no risk of confusion.

\section{The Green function for the periodic problem}\label{section:green_function}
Let $B$ be a Banach space and $B^*$ be its dual space. We will denote the action of the dual space by
$$ \langle f , u \rangle_B := f(u) \quad f\in B^* , u \in B. $$
We recall the definition of the planar Fourier transform. Formally, it is defined for $u:\R^2 \to \R$ by
$$ \cF[u](\xi') := \frac{1}{2\pi} \int_{\R^2} u(x')e^{-i x' \cdot \xi'} dx' \quad \xi' \in \R^2.$$
More generally, for functions $w:\Omega \to \R$, $(x',x_3) \mapsto w(x',x_3)$, we denote the planar Fourier transform with respect to $x'$ by
$$ \cF_{x'}[w(\cdot,x_3)] (\xi') := \frac{1}{2\pi} \int_{\R^2} w(x',x_3)e^{-i x' \xi'} dx' \qquad \xi' \in \R^2, \, x_3 \in [-\ell,\ell].$$ 
To keep the notation as consistent as possible, we use the symbol $\cF$ exclusively to denote the planar Fourier transform of a function $\R^2 \to \R$ and $\cF_{x'}$ exclusively for the planar Fourier transform of a function $\R^3 \to \R$.

In order to simplify the notation, we formally define the operator $\cF_*$ which transforms functions on $\overline{\Omega}$ to functions on $\overline{\Omega}$ by 
$$
[\cF_* u](\xi) = \mathcal{F}_{x'}u(\cdot,\xi_3)(\xi') \qquad \text{for $\xi= (\xi',\xi_3) \in \overline{\Omega}$.}
$$
Our starting point is to compute the Fourier transform of
\begin{align*}
	\begin{cases}
		\Delta w(x) = \delta(x-y) \quad &x \in \Omega,\\
		w(\cdot,-\ell) = w(\cdot,\ell) \quad &\text{in } \R^2, \\
		\partial_3 w(\cdot,-\ell) = \partial_3 w(\cdot,\ell) \quad &\text{in } \R^2,
	\end{cases}
\end{align*}
in the variable $x'$. For any fixed $y \in \Omega$ this yields the family of one-dimensional ODE problems
\begin{align}\label{eq:ode}
	\begin{cases}
		\partial_{3}^2 w(\xi',x_3) - |\xi'|^2 w(\xi',x_3) = \frac{e^{-i\xi'y'}}{2 \pi}\delta(x_3-y_3) &\qquad x_3 \in (-\ell,\ell),\\
		w(\xi',0) = w(\xi',1), \\
		\partial_{3} w(\xi',0) = \partial_3 w(\xi',1),
	\end{cases}	
\end{align}
for the transformed function 
$$
w(\xi',x_3)= \mathcal{F}_{x'}u(\cdot,x_3)(\xi') \qquad \text{for $\xi' \in \R^2, x_3 \in [-\ell,\ell].$}
$$

Let us analyze the one-dimensional problem
\begin{equation}\label{eq:greenode}
	\begin{cases}
		u''(t)- r^2 u(t)  = f(t)/2 \pi \qquad t \in (-\ell,\ell),\\
		u(0) = u(1),\\
		u'(0) = u'(1),
	\end{cases}	
\end{equation}
where $r \in \R \setminus\{0\}$ and $f \in H^1_{p}(-\ell,\ell) := \{ u \in H^1(-\ell,\ell) : u(-\ell)=u(\ell) \}$. Since we have the embedding $W^{1,2}(-\ell,\ell) \hookrightarrow C^{0,\frac{1}{2}}(-\ell,\ell)$\footnote{see e.g. \cite[Theorem 7.13]{leo}}, the equality $u(-\ell)=u(\ell)$ is meant in the normal sense for the continuous representative of $u$. We define the function
\begin{equation*}
	h_r(t,s)  = -\frac{e^{r|t-s|} + e^{r(2\ell-|t-s|)}}{4 \pi r (e^{2\ell r}-1)},
\end{equation*}
which is the Green function of the problem \eqref{eq:greenode}, as we will see below. Observe that $h_r$ is symmetric in $s$ and $t$, and for all $s \in [-\ell,\ell]$ the map $t \mapsto h_r(t,s)$ is absolutely continuous in $[-\ell,\ell]$ with derivative 
$$\frac{d}{dt}h_r(t,s) = \begin{cases}
	\text{sgn}(s-t) \frac{e^{r|t-s|}-e^{r(2\ell-|t-s|)}}{4 \pi (e^{2\ell r}-1)} \quad &t \neq s\\
	0\quad &t = s.
\end{cases}$$
Then $\partial_t h_r(\cdot,s) \in L^2(-\ell,\ell)$ for all $s \in [-\ell,\ell]$ and it follows that $h_r(\cdot,s) \in H^1(-\ell,\ell)$ for all $s \in [-\ell,\ell]$. A simple computation shows that $h_r(t,\cdot)$ satisfies the periodic boundary conditions in \eqref{eq:greenode} for all $t \in [-\ell,\ell] \to \R$, so $h_r(t,\cdot) \in H^1_{p}(-\ell,\ell)$ for all $t \in [-\ell,\ell]$.

We introduce, for any measurable function $f :[-\ell,\ell] \to \R$, the integral transform
\begin{equation}\label{eq:operatorK}
	\mathcal{H}_rf(t) := \int_{-\ell}^\ell h_r(t,s)f(s)ds \in [-\infty,\infty] \qquad \text{for all } t \in [-\ell,\ell]
\end{equation}
\begin{lemma}\label{lemma:greenode}
	For all $r \in \R \setminus \{0\}$ and $f \in H^1_{p}(-\ell,\ell)$ the boundary value problem \eqref{eq:greenode}
	has a unique solution in $H^1_p(-\ell,\ell)$ given by $u:= \mathcal{H}_rf$. Moreover, $u \in C^{2,\frac{1}{2}}([-\ell,\ell])$ and it is a classical solution to \eqref{eq:greenode}.
	
\end{lemma}
\begin{proof}
	For the remaining of the proof we fix $r \neq 0$ and write $h(t,s) := h_r(t,s)$ for short. Since $f \in H^1(-\ell,\ell)$ and both $h$ and $\partial_t h$ are bounded in $[-\ell,\ell]^2$, it is easy to see from Lebesgue's theorem that $w \in W^{1,2}(-\ell,\ell)$ and
	\begin{align*}
		\partial_t u(t) &= -\partial_t \int_{-\ell}^{\ell} \frac{e^{r|t-s|} + e^{r(2\ell-|t-s|)}}{4 \pi r (e^{2\ell r}-1)} f(s) ds = -\partial_t \int_{-\ell}^{\ell} \frac{e^{r|s|} + e^{r(2\ell-|s|)}}{4 \pi r (e^{2\ell r}-1)} f(t-s) ds \\
		&= \int_{-\ell}^\ell h(t,s)f'(s) ds
	\end{align*}
	and similarly, we have
	$$ \partial_t^2 u(t) = \int_{-\ell}^\ell \partial_t h(t,s)f'(s) ds. $$
	Thus $u \in W^{2,2}(-\ell,\ell)$ and from the Sobolev embeddings it follows that $u \in C^{1,\frac{1}{2}}(-\ell,\ell)$.
	
	We proceed to show that $u$ solves the equation in the weak sense. A simple computation shows that leaving $s \in [-\ell,\ell]$ fixed, the map $t \mapsto h(t,s)$ solves the homogeneous equation 
	$$ u''(t) - r^2 u(t)=0,$$
	in the subintervals $(-\ell,s)$ and $(s,\ell)$. Further, fixing $s \in (-\ell,\ell)$ the right and left side derivatives of the map $t \mapsto h(t,s)$ exist at $s$ and are equal to 
	$$\partial_t h(s^-,s) =-1/4 \pi \quad \text{and} \quad \partial_t h(s^+,s) =1/4 \pi.$$
	Now, take $\varphi \in C^2([-\ell,\ell])$ with $\varphi(-\ell)=\varphi(\ell)$ and integrate by parts
	\begin{align*}
		&\int_{-\ell}^\ell u'(t) \varphi'(t) dt = \int_{-\ell}^\ell \int_{-\ell}^\ell h(t,s)f'(s)\varphi'(t) ds dt \\
		&= \int_{-\ell}^\ell ( h(t,t)f(t)-h(t,-\ell)f(-\ell)+h(t,\ell)f(\ell)-h(t,t)f(t))\varphi'(t)dt\\
		&\qquad + \int_{-\ell}^\ell \int_{-\ell}^\ell \partial_s h(t,s)f(s)\varphi'(t) dt ds = \int_{-\ell}^\ell \int_{-\ell}^\ell \partial_t h(t,s)f(s)\varphi'(t) dt ds\\
		&= \int_{-\ell}^\ell ( \partial_t h(s^-,s)\varphi(s)-\partial_th(-\ell,s)\varphi(-\ell)+ \partial_t h(\ell,s)\varphi(\ell)- \partial_t h(s^+,s)\varphi(s))f(s)ds\\
		&\qquad  - \int_{-\ell}^\ell \int_{-\ell}^\ell \partial_t^2 h (t,s)f(s)\varphi(t) ds dt \\
		&= -\int_{-\ell}^\ell \frac{\varphi(s)f(s)}{2 \pi} ds - \int_{-\ell}^\ell \int_{-\ell}^\ell r^2 h(t,s)f(s)\varphi(t) dt ds \\
		& = \int_{-\ell}^\ell \left( -\frac{f(t)}{2 \pi} - r^2 w(t) \right) \varphi(t) dt 
	\end{align*}
	Thus, we have $u''(t) - r^2 u(t) = f(t)/4 \pi$ a.e. in $[-\ell,\ell]$. Moreover, since $f \in C^{0,\frac{1}{2}}(-\ell,\ell)$ it follows that $u'' \in C^{0,\frac{1}{2}}(-\ell,\ell)$, i.e. $u \in C^{2,\frac{1}{2}}(-\ell,\ell)$ and $u$ is the classical solution to \eqref{eq:greenode} and, since $k$ satisfies $h(-\ell,s) = h(\ell,s)$ for all $s \in [-\ell,\ell]$, we see that $u$ satisfies the periodic conditions \eqref{eq:greenode}.
\end{proof}

Thus, for all $\xi' \in \Omega$ the solution to \eqref{eq:ode} is given by
\begin{equation}\label{eq:w_function}
	w(\xi',x_3) = -e^{-i \xi' y'}\frac{e^{|\xi'|x_3-y_3|} + e^{|\xi'|(2\ell-|x_3-y_3|)}}{4 \pi |\xi'|(e^{2\ell|\xi'|}-1)}
\end{equation}

Now, we start with the proof of Theorem \ref{thm:asymtoticsGreen}

\begin{proof} [Proof of Theorem \ref{thm:asymtoticsGreen} (Item 1)]
In what follows we will set $y = 0$, i.e. we find $K(\cdot,0)$, and we obtain the general expression using the translation rules of the Fourier transform. 

Observe that
\begin{align}\label{eq:split}
	\frac{e^{|\xi'|t} + e^{|\xi'|(2\ell-t)}}{4 \pi |\xi'|(e^{2\ell|\xi'|}-1)} = \frac{e^{-|\xi'|t}}{4\pi|\xi'|} + \frac{\cosh(|\xi'|t)}{4\pi|\xi'|(e^{2\ell|\xi'|}-1)} \quad \forall t \in\R,\; \xi' \in \R^2.
\end{align}

We continue the analysis of this function in parts:\\

\textit{Part 1. The Fourier transform of $\xi' \mapsto e^{-|\xi'|t}/|\xi'|$\footnote{Here, we note that for an even real-valued function the Fourier transform is the same as the inverse Fourier transform.}\\
}
The inverse Fourier transform of the first term on the RHS of \eqref{eq:split} can be computed fixing $t \in \R$ and using the Hankel transform of order 0 of the map $r \mapsto \frac{e^{-r|t|}}{r}$
\begin{equation}\label{eq:hankel}
	\mathcal{F}_{x'}\left(\frac{1}{4 \pi |(\cdot,t)|}\right)(\xi') = \frac{e^{-|\xi'|t}}{4 \pi |\xi'|} \qquad \text{for $\xi' \in \R^2,\;t \in \R$.}
\end{equation}
Thus, we have found the singular part of $K(\cdot,0)$. We set $\tilde{K}_1(x) = -(4\pi |x|)^{-1}$, which then satisfies
\begin{equation*}
	\mathcal{F}_{*}\tilde{K}_1(\xi) = -\frac{e^{-|\xi'|\xi_3}}{4 \pi |\xi'|} \qquad \text{for $\xi \in \R^3$.}
\end{equation*}
\\

\textit{Part 2. The Fourier transform of $\cosh(|\xi'|t)/4 \pi |\xi'|(e^{2\ell|\xi'|}-1)$}

Fix a nonnegative cut-off function $\chi \in C^\infty_c(\mathbb{R}^2)$ with $|\chi| \leq 1$ which equals $1$ in a neighborhood of $0$ and we extend it as a constant function to $\Omega$
$$ \xi \mapsto \chi(\xi') \qquad \xi = (\xi',\xi_3) \in \Omega .$$
To split the behavior of this function at $0$ and infinity, we write
\begin{align*} w(\xi',t) &:= \frac{\cosh(|\xi'|t)}{4 \pi |\xi'|(e^{2 \ell |\xi'|}-1)} = w(\xi',t) \chi(\xi') + w(\xi',t)(1-\chi(\xi))\\
	&=:w_1(\xi',t) + w_2(\xi',t),
\end{align*}
and now we analyze each term individually.

\textit{Part 2.1 The $w_1$ term}

We use the Laurent series expansion of the map $z \mapsto \frac{ (\cosh(t z))}{z(e^{2\ell z}-1)}$ at zero to express the $\hat{w}_1$ term as
\begin{equation}\label{eq:eq5}
	w_1(\xi',t) = \chi(\xi') \left( \frac{1}{8 \ell \pi |\xi'|^2} -\frac{ 1}{8 \pi |\xi'|} +\frac{2 \ell^2+3t^2}{24 \ell \pi} - r(\xi',t) \right).
\end{equation}
Now, we will show that the term
$$ \frac{\chi(\xi')}{8 \ell \pi |\xi'|^2}, $$
will produce the logarithmic behavior of the Fourier transform at infinity, whereas the term
$$ \chi(\xi') \left( \frac{2 \ell^2 + 3t^2}{24 \ell \pi} - \frac{1}{8 \pi |\xi'|} - r(\xi',t) \right),$$
converges to $0$ as $|x| \to \infty$ for $x \in \Omega$.

Observe that the remainder term $r(\xi',t)=\hat{w}(\xi) -\frac{1}{8 \ell \pi |\xi'|^2}+\frac{1}{8 \pi |\xi'|}-\frac{2\ell ^2+3t^2}{24 \ell \pi}$ is continuous in $(\xi',t)$. Then, for fixed $t$, the map $\xi' \mapsto r(\xi',t)\chi(\xi',t)$ is an element of $C_c(\mathbb{R}^2)$. Thus, its Fourier transform belongs to $C^\infty(\mathbb{R}^2)$ and by the Riemann-Lebesgue lemma we have, for $t \in [-\ell,\ell]$ fixed
\begin{equation}\label{eq:fresidue0} \mathcal{F}_*[\chi r ](x',t) \to 0 \quad \text{as } |x'|\to \infty, \, x' \in \R^2.
\end{equation}
Now, we will show below that this convergence is locally uniform on $t \in [-\ell,\ell]$ and thus, by compactness, uniform in $t \in [-\ell,\ell]$. Indeed,
we observe that the map $\xi \mapsto \chi(\xi)r(\xi)$ is continuous and is compactly supported in $\mathbb{R}^2 \times [-\ell,\ell]$, thereby it is bounded by a constant $C > 0$. Thus
$$ |\chi(\xi)[r(\xi',t) - r(\xi',t^*)]| \leq 2C 1_{\{\chi > 0\}}(\xi') \in L^1(\R^2) \quad \text{for all } t,t^* \in [-\ell,\ell]. $$
Then, from the dominated convergence theorem we conclude that
\begin{align*}
	|\mathcal{F}[\chi r](x',t)-\mathcal{F}[\chi r](x',t^*)| \leq \frac{1}{2 \pi} \int_{\mathbb{R}^2} \chi(\xi')|r(\xi',t)-r(\xi',t^*)| d\xi' = o(1),
\end{align*}
uniformly on $x' \in \R^2$ as $|t-t^*| \to 0$.
Thus we have shown that
$$  \mathcal{F}_*[ \chi r](x) \to 0 \quad \text{as }|x| \to \infty, \, \text{ uniformly for } x \in \Omega. $$

We carry out a similar argument for
$$ \chi(\xi) \left(\frac{2\ell^2+3|\xi_3|^2}{24 \ell \pi} - \frac{1}{8 \pi |\xi'|}  \right), $$
this time using the fact that this function lies in $L^1(\R^2)$ as a function of $\xi'$ and is compactly supported in $\R^2 \times [-\ell,\ell]$. Thus, its Fourier transform is a smooth function that converges to $0$ uniformly as $|x| \to \infty$ for $x \in \Omega$.

To find the transform of the first term in \eqref{eq:eq5}, we interpret the map
$$ \R^2 \to \R : \xi' \mapsto 1/|\xi'|^2,$$
as the tempered distribution
$$ \langle |z|^{-2}_m, \varphi(z) \rangle_{\mathcal{S}(\R^2)} := \int_{|z|<m}  \frac{\varphi(z) - \varphi(0)}{|z|^2}dx + \int_{|z|>m}\frac{\varphi(z)}{|z|^2}dz \quad \forall \varphi \in \mathcal{S}(\R^2), \, m>0$$
where $\mathcal{S}(\R^2)$ denotes the set of Schwartz functions in $\R^2$, see \cite{ss}, and $m$ is to be set appropriately. These distributions are related to each other in function of $m$ through the identity 
\begin{equation}\label{eq:identity_m}
	|z|_{m_2}^{-2} =  |z|_{m_1}^{-2} -c_m \delta \quad \text{for } m_2 > m_1, \quad \text{where } c_m := \int_{m_1 < |z| < m_2}\frac{dz}{|z|^2}.
\end{equation}
In this way, we have the following expression for the Fourier transform of the map \mbox{$\R^2 \to \R : z \mapsto |z|^{-2}$} in the distributional sense for $m = 1$
$$ \mathcal{F}\left( {\frac{|z|^{-2}_1}{8 \ell \pi}} \right) =  -\frac{\log|z|+\log \pi + \gamma}{4 \ell }, $$
where $\gamma$ is the Euler-Mascheroni constant, see \cite[Chapter 32]{d}. To simplify the expressions, using \eqref{eq:identity_m}, we may choose $m > 0$ in such a way that
$$ \mathcal{F}\left( {\frac{|z|^{-2}_m}{8 \pi}} \right) =  -\frac{\log|z|}{4}, $$
and we will interpret the function $1/|z|^2$ as this particular distribution, unless otherwise stated.

Now, since $\chi \in \mathcal{S}(\R^2)$, we have $\cF_* [\chi](\cdot,t) \in \mathcal{S}(\R^2)$ for all $t \in [-\ell,\ell]$, so that the convolution 
$$ (\cF_*[\chi] \ast \log|\cdot|) (x',t) := \int_\Omega \log|x'-y'|\cF_*[\chi](y',t) dy' \qquad \text{ for all } (x',t) \in \Omega. $$
is well-defined and constant on the third variable. Then, the properties of the Fourier transform imply that the planar transform of this convolution is $-\frac{\chi(\xi',t)}{8 \ell \pi |\xi'|^2}$ for all $t \in \R$.\footnote{Here, we may assume without loss of generality that $\chi$ is an even function so that $\cF \cF \chi(\xi') = \chi(-\xi')=\chi(\xi')$.}

In summary, we have shown that
$$ x= (x',x_3) \mapsto \cF_* w_1(x) = -(\cF_*[\chi] \ast \log|\cdot|)(x) + o(1) $$
as $|x| \to \infty$ for $x \in \Omega$.

\textit{Part 2.2. The asymptotics of the logarithmic convolution}

We have the following identities
\begin{equation}\label{eq:log}
	\begin{gathered}
		\log|x-z| = \log(1+|x-z|) - \log(1 + \frac{1}{|x-z|})\\
		\log(1 + |x-z|) \leq \log(1+|x|) + \log(1+|z|) \quad \text{and} \quad \log\left(1+\frac{1}{|x-z|}\right) \leq \frac{1}{|x-z|}.
	\end{gathered}
\end{equation}
In this way, recalling that both $\chi$ and $\cF \chi$ are constant on the third variable, it follows that
\begin{align}\label{eq:eq2}
	\begin{split}
		&(\cF_* [\chi]\ast \log|\cdot|)(x',t) =  \int_{\Omega}\cF_* [\chi](y',t)\log|x'-y'|dy  \\
		&= 2\ell \int_{\R^2} \log(1+|x'-y'|) \cF [\chi](y')dy' - 2 \ell \int_{\R^2} \log\left(1+ \frac{1}{|x'-y'|}\right)\cF [\chi](y') dy.
	\end{split}
\end{align}
Now, we estimate the last integral
\begin{align}\label{eq:eq3}
	\begin{split}
		&\int_{\R^2} \log \left(1+ \frac{1}{|x'-y'|} \right) \cF [\chi](y') dy' \leq \int_{\R^2}\frac{|\cF [\chi](y')|}{|x'-y'|}dy' = \int_{|x'-y'|<1}\frac{|\cF [\chi](y')|}{|x'-y'|}dy'  \\
		&\quad+ \int_{|x'-y'|>1}\frac{|\cF [\chi](y')|}{|x'-y'|}dy' \leq |\cF [\chi]|_{\infty} \int_{|y'|<1}\frac{1}{|y'|}dy' + \int_{\R^2}|\cF [\chi](y')|dy'.
	\end{split}
\end{align}
Observe that the right-hand side of \eqref{eq:eq3} is a constant that depends only on $\cF [\chi]$. Then, for the remaining integral in \eqref{eq:eq2} we use the inequalities \eqref{eq:log} to estimate as follows
\begin{align*}
	\frac{\log(1+|x'-y'|)}{\log(1+|x|)}\cF [\chi](y') &\leq \left( \frac{\log(1+|x'|)}{\log(1+|x|)} + \frac{\log(1+|y'|)}{\log(1+|x|)} \right) \cF [\chi](y') \\
	&\leq C(1 + \log(1+|y'|))\cF [\chi](y') \in L^1(\R^2) 
\end{align*}
uniformly on $|x| \geq 1 $ and a constant $C>0$. Now, we divide in \eqref{eq:eq2} by $\log(1+|x|)$ and pass to the limit using the dominated convergence theorem 
\begin{align*}
	\frac{(\hat{\chi} \ast \log|\cdot|)(x)}{\log(1+|x|)} \to 2 \ell \int_{\R^2} \cF [\chi](y)dy = 2 \ell \mathcal{F} [\mathcal{F} [\chi]](0) = 2 \ell \chi(0) = 2 \ell >0 \quad \text{ as } |x|\to \infty, x \in \Omega.
\end{align*}
Altogether this means that $x \mapsto (\cF_* [\chi] \ast \log|\cdot|) (x)$ is a smooth function with 
$$(\cF_* [\chi]\ast\log|\cdot|)(x) \sim \log(1+|x|) \quad \text{as }|x| \to \infty, x \in \Omega.$$

Thus far we have found the planar Fourier transform of $w_1$ and it satisfies $\cF_* w_1 \in C^\infty(\overline{\Omega})$ with the asymptotic profile 
$$ \cF_* [w_1](x) \sim -\log(1+|x|) \quad \text{as } |x| \to \infty,x\in \Omega.$$

\textit{Part 2.3. The $w_2$ term}

We have the estimate
\begin{align}\label{eq:coshdecay} \left|  (1-\chi(\xi',t)) \frac{\cosh(|\xi'|t)}{4\pi|\xi'|(e^{2\ell|\xi'|}-1)} \right| \leq \left| \frac{\cosh(\ell|\xi'|)}{4 \pi |\xi'|(e^{2 \ell |\xi'|}-1)} \right|, \quad (\xi',t) \in \Omega
\end{align}
where the RHS of \eqref{eq:coshdecay} and its derivatives decay exponentially on $|\xi'|$. Then it is easy to see that $w_2 \in \mathcal{S}(\R^2)$ as a function of $\xi'$ for all $t \in [-\ell,\ell]$. Moreover, we can adapt the arguments above using the bound \eqref{eq:coshdecay} to show that $\cF_* [w_2](x) \to 0$ uniformly as $|x| \to \infty$.

Now, we set $\tilde{K}_2(x) = -\mathcal{F}_*(w_1 + w_2)(x)$. Then, we have $\tilde{K}_2 \in C^\infty(\Omega)$ and
\begin{equation}\label{eq:K2tilde}
	\tilde{K}_2(x) \sim \log(1+|x|) \quad \text{as } |x|\to \infty.
\end{equation}
\newline

Now, we set $\tilde{K} := \tilde{K}_1 + \tilde{K}_2$. To obtain the $K(x,y)$ for general $y \in \overline{\Omega}$, we observe from the properties of the Fourier transform, that for all $y \in \Omega$
$$ \mathcal{F}_*[ \tilde{K}(\cdot-y)](\xi) = e^{-i \xi' y'} \mathcal{F}_* [\tilde{K}](\xi',\xi_3 - y_3) = -e^{-iy'\xi'}\frac{  e^{|\xi'||\xi_3-y_3|} + e^{|\xi'|(2\ell-|\xi_3-y_3|)} }{4 \pi |\xi'|(e^{2\ell|\xi'|}-1)} \quad \xi \in \Omega. $$
Thus 
\begin{equation}\label{eq:translation_K}
	K(x,y) = K(x-y,0) = \tilde{K}(x-y) \quad \forall x,y \in \overline{\Omega},
\end{equation}
and similarly we have $K_1(x,y) = K_1(x-y,0) = \tilde{K}_1(x-y)$ and $K_2(x,y) = K_2(x-y,0) = \tilde{K}_2(x-y)$ for all $x,y \in \overline{\Omega}$. As a consequence, we have $K_1(x,y) = -(4 \pi |x-y|)^{-1}$, $K_2 \in C^{\infty}(\overline{\Omega}^2)$ and, since the asymptotics \eqref{eq:K2tilde} hold uniformly as $|x|\to \infty$, we have
$$ K_2(x,y) \sim \log(1+|x-y|) \qquad \text{uniformly as } |x-y| \to \infty, \, x,y \in \Omega.$$ 
In conclusion, we have shown that there is a function $K \in C^\infty(\overline{\Omega}^2 \setminus \{(x,y) \in \Omega^2 : x=y\})$ such that, for all $y \in \Omega$ the following identity holds
$$ \mathcal{F}_*[K(\cdot,y)](\xi) = \frac{-e^{-iy'\xi'}}{4 \pi |\xi'|(e^{2\ell|\xi'|}-1)} ( e^{|\xi'||\xi_3-y_3|} + e^{|\xi'|(2\ell-|\xi_3-y_3|)}) \quad \xi \in \Omega. $$
in the sense that for all $y \in \Omega, t \in [-\ell,\ell]$ and $\varphi \in \mathcal{S}(\R^2)$ the following holds
\begin{equation}\label{eq:distributiondefinition}
	\langle \mathcal{F}_*[K((\cdot,t),y)] , \varphi \rangle_{\mathcal{S}(\R^2)} = -\int_{\R^2} \frac{-e^{-iy'\xi'} ( e^{|\xi'||t-y_3|} + e^{|\xi'|(2\ell-|t-y_3|)}) }{4 \pi |\xi'|(e^{2\ell|\xi'|}-1)} (\varphi(\xi') - \varphi(0)1_{|\xi'|\leq m}) d\xi',
\end{equation}
where $m$ is chosen so that the planar Fourier transform of the distribution $z \mapsto |z|_m^{-2}$ satisfies
$$ \mathcal{F}\left( {\frac{|z|^{-2}_m}{8 \pi}} \right) =  -\frac{\log|z|}{4 }. $$

Finally, since
$$\xi \mapsto \frac{-e^{-iy'\xi'}}{4 \pi |\xi'|(e^{2\ell|\xi'|}-1)} ( e^{|\xi'||\xi_3-y_3|} + e^{|\xi'|(2\ell-|\xi_3-y_3|)}) \qquad y \in \Omega,$$
satisfies the periodic boundary conditions and is symmetric in $\xi$ and $y$, then the Fourier transform $K$ satisfies $K(\cdot,-\ell,y) = K(\cdot,\ell,y)$ for all $y \in \Omega$, and $K(x,y) = K(y,x)$ for all $x,y \in \overline{\Omega}$, $x \neq y$. This concludes with points $1$ and $2$ of Theorem \ref{thm:asymtoticsGreen}. We will complete the proof of item $2$ below.
\end{proof}
From the identity \eqref{eq:translation_K} we obtain
\begin{corollary}\label{corollary:translation_K}
	For all $x,y,z \in \overline{\Omega}$ the following identity holds
	$$ K(x,y+z) = K(x-z,y). $$ 
\end{corollary}

\begin{remark}\label{remark:derivativeaG}
With a similar analysis we obtain
$$\frac{d}{d x_i} K_2(x,y) = \frac{d}{d y_i}K_2(x,y) \to 0 \quad \text{as } |x-y| \to \infty,$$
uniformly on $x,y \in \Omega$, $i=1,2,3$. This will be useful to prove that $\mathcal{K}[u^2]$ belongs to an appropriate local Sobolev space.

This can be verified by subtracting from $K$ the singular part $K_1$ and apply the Fourier transform to obtain
$$ \mathcal{F}_*[\tilde{K} - \tilde{K}_1] (\xi) = - \frac{\cosh(|\xi'||\xi_3|)}{4 \pi |\xi'|(e^{2 \ell |\xi'|}-1)} . $$
Thus
$$  \mathcal{F}_*[\partial_i \tilde{K}_2] (\xi) = -\frac{i \xi_j \cosh(|\xi'||\xi_3|)}{4 \pi |\xi'| (e^{2 \ell |\xi'|}-1)} \quad \text{for } j =1,2$$
and
$$ \mathcal{F}_*[\partial_3 \tilde{K}_2] (\xi) = -\sgn(\xi_3) \frac{ \sinh(|\xi'||\xi_3|)}{4 \pi (e^{2 \ell |\xi'|}-1)},
$$
and we may proceed as above.
\end{remark}

\section{The Poisson problem}\label{sec:poisson}
In this section we will prove some results regarding Poisson part of the system \eqref{eq:SP_system_periodic}. We introduce the space
\begin{equation}\label{eq:Sobolev_periodic}
	H^1_{p}(\Omega):=\{ u \in H^1(\Omega) : u(\cdot,-\ell) = u(\cdot,\ell) \},
\end{equation}
here the equality $u(\cdot,-\ell) = u(\cdot,\ell)$ should be understood in the sense of traces, see e.g. \cite[Chapter 15]{leo}. Observe that the periodic condition $u(\cdot,-\ell) = u(\cdot,\ell)$ implies that the periodic extension of $u$ to $\R^3$ satisfies $\overline{u} \in H^1_{loc}(\R^3)$. Moreover, for all $\tau \in \R^3$ we can define an isometric translation operator $T_\tau : H^1_p(\Omega) \to H_p^{1}(\Omega)$ by setting
\begin{equation}\label{eq:translation_op}
	T_\tau w(x) := \overline{w}(x-\tau) \qquad x \in \Omega, w \in H^1_p(\Omega),
\end{equation}
where $\overline{w}:\R^3 \to \R$ is the periodic extension of $w$.
Then we consider the subspace 
$$ X := \{ u \in H^1_{p}(\Omega) : |u|_* < \infty \},$$
equipped with the norm
$$ \| u \|_X^2 := \| u \|_{H^1(\R^2)}^2 + |u|_*^2, $$
where
$$ |u|_*^2 := \int_\Omega \log(1+|x|)u^2(x) dx. $$
For any  measurable function $u : \Omega \to \R$ we define the convolution type operator
$$ \cK [u](x) := \int_{\Omega} K(x,y) u(y) dy \in [-\infty,\infty] \qquad x \in \Omega,$$
and if $u \in X$ we have the following
\begin{lemma}\label{lemma:convolution}
For all $u \in X$ we have $\cK [u^2] \in L^{\infty}_{loc}(\Omega)$, $\nabla (\cK [u^2]) \in L^2_{loc}(\Omega)$. Moreover, if $u \neq 0$ we have the asymptotic description 
$$\cK [u^2](x) \sim \log(1+|x|) \quad \text{as } |x| \to \infty.$$
As a consequence $\cK [u^2] \in W^{1,2}_{loc}(\Omega)$.
\end{lemma}
\begin{proof}
	From the asymptotic description of $K$ in Theorem \ref{thm:asymtoticsGreen} we know the limit 
	$$ l := \lim_{|x-y| \to \infty} \frac{K(x,y)}{\log(1+|x-y|)} > 0, $$
	exists and there are $ R>0 $ and a constant $C>0$ such that for all $x,y \in \Omega$
	$$ |x-y||K(x,y)| < C_1 \quad \text{for } |x-y|<R \quad  \text{and} \quad \frac{|K(x,y)|}{\log(1+|x-y|)} < C_1 \quad \text{for } |x-y| \geq R. $$
	Moreover, we choose $R>0$ in such a way that
	$$ \int_{|x|>R} u^2(x)dx \geq \frac{|u|_2^2}{2} .$$
	On the other hand we use \eqref{eq:log} to derive, for $x \in \Omega$
	\begin{align}\label{eq:boundconv}
		\begin{split}
			&|( \cK [u^2] )(x)| = \left| \int_{\Omega} K(x,y)u^2(y)dy \right| = \left| \left(\int\displaylimits_{|x-y| < R}  +\int\displaylimits_{R<|x-y|} \right)  K(x,y) u^2(y) dy \right| \\
			&\quad \leq C_1 \int_{|x-y|<R} \frac{u^2(y)}{|x-y|}dy + C_1 \int_{R<|x-y|}\log(1+|x-y|)u^2(y) dy \\
			&\quad \leq C_1 \int_{|x-y|<R}\frac{u^2(y)}{|x-y|}dy + C_1 \int_{R<|x-y|} \left[ \log(1+|x|) + \log(1+|y|) \right] u^2(y)dy \\
			&\quad \leq  C_1\int_{|x-y|<R}\frac{u^2(y)}{|x-y|}dy +C_1\log(1+|x|)|u|_2^2 + C_1 |u|_*^2  \\
			&\quad \leq C_1 \left( \int_{|x-y|<R}\frac{dy}{|x-y|^2} \right)^{\frac{1}{2}} |u|_4^2 + C_1\log(1+|x|)|u|_2^2 + C_1 |u|_*^2 \\
			&\quad \leq C_2(1+\log(1+|x|))\|u\|_X^2
		\end{split}
	\end{align}
	This shows that $\cK u^2 \in L_{loc}^\infty(\Omega)$. We also extract the inequality
	\begin{equation}\label{eq:int_less_R}
	\left| \int_{|x-y|<R}K(x,y)u^2(y) dy \right| \leq C |u|_4^2.
	\end{equation}
	Now, we prove the asymptotics. Assume that $u \neq 0$. Then we observe that, for all $y \in \Omega$ and $|x|\geq 1$ we have
	\begin{align*}
		&\left|\frac{K(x,y)}{\log(1+|x|)} \right|1_{\{|x-y|>R\}} = \left|\frac{K(x,y)}{\log(1+|x-y|)} \frac{\log(1+|x-y|)}{\log(1+|x|)} \right| 1_{\{|x-y|>R\}} \\
		&\qquad \leq C_1 \frac{\log(1+|x|) + \log(1+|y|)}{\log(1+|x|)} 1_{\{|x-y|>R\}} \leq C_2(1 + \log(1+|y|))1_{\{|x-y|>R\}},
	\end{align*}
	then, we multiply by $u^2$ to obtain
	$$ \left|\frac{K(x,y)}{\log(1+|x|)} u^2(y) \right|1_{\{|x-y|>R\}} \leq C_2(1+\log(1+|y|))u^2(y) \in L^1(\Omega).$$
	Thereby, by the dominated convergence theorem, we see that
	\begin{align*}
		\lim_{|x| \to \infty} \int_{R < |x-y|} \frac{ K(x,y)u^2(y)}{\log(1+|x|)} dy &= \lim_{|x| \to \infty} \int_{R < |x-y|} \frac{K(x,y)}{\log|x-y|}\frac{\log|x-y|}{\log(1+|x|)} u^2(y) dy\\
		 &= l \int_{R <|x-y|} u^2(y) dy.
	\end{align*}
	Then, together with \eqref{eq:int_less_R} this means that
	$$ \lim_{|x| \to \infty} \frac{\cK [u^2](x)}{\log(1+|x|)} = \lim_{|x| \to \infty} \int_{R < |x-y|} \frac{ K(x,y)u^2(y)}{\log(1+|x|)} dy \geq \frac{l |u|_2^2}{2} >0.$$
	
	Now, we show that $\cK [u^2]$ is weakly differentiable and $\nabla(\cK[ u^2]) \in L^2_{loc}(\Omega)$. First we recall the split $K = K_1 + K_2$, where $K_2 \in C^\infty(\overline{\Omega}^2)$ and $K_1(x,y) = -(4 \pi |x-y|)^{-1}$ .
	
	First, we observe that
	$$ (\cK_1 [u^2]) (x) := \int_\Omega K_1(x,y)u^2(y) dy = -\frac{1}{4 \pi }\int_\Omega \frac{u^2(y)}{|x-y|} dy, $$
	so that $\cK_1 [u^2](x)$ is weakly differentiable and
	$$ \partial_i (\cK_1 [u^2])(x) = \frac{1}{4 \pi }\int_\Omega \frac{(x_i-y_i) u^2(y)}{|x-y|^3} dy, $$
	and thus, by the Hardy-Littlewood-Sobolev inequality
	\begin{align*}
		\left| \int_\Omega \left(\int_\Omega \frac{(x_i-y_i)u^2(y)}{|x-y|^3} dy\right)^2 dx \right| &\leq \int_\Omega \left(\int_\Omega \frac{u^2(y)}{|x-y|^2}dy\right)^2dx = \int_\Omega (I_1 \ast u^2|_\Omega)^2(x) dx \\
		&\leq \int_{\R^3}(I_1 \ast u^2|_\Omega)^2(x) dx \leq C |u|_{12/5}^4,
	\end{align*}
	where 
	$$u^2|_\Omega(x) := \begin{cases}
		u^2(x) \quad \text{if } x \in \Omega,\\
		0 \quad \text{otherwise,}
	\end{cases}
	$$
	and $I_1$ denotes the Riesz potential $I_\alpha$ for $\alpha = 1$, see \cite{hlp}.
	Thus $\partial_i (\cK_1 [u^2]) \in L^2(\Omega)$. Similarly, we define $\cK_2 [u^2]$. Now, we recall from Remark \ref{remark:derivativeaG} that 
	$$\frac{d}{dx_i} K_2(x,y) \to 0 \quad \text{as } |x-y| \to \infty,$$
	and since $K_2 \in C^\infty(\overline{\Omega}^2)$, it follows that
	$$ \partial_i (\cK_2 [u^2])(x)= \int_\Omega \frac{d}{dx_{i}} K_2(x,y)u^2(y) dy.$$
	Moreover, since $\frac{d}{d x_i} K_2(x,y)$ is continuous over $\overline{\Omega}^2$ and $\frac{d}{d x_i} K_2(x,y) \to 0$ as $|x-y| \to \infty$ from Remark \ref{remark:derivativeaG}, we have the estimate
	$$ \left| \int_\Omega \frac{d}{dx_i} K_2(x,y)u^2(y) dy \right| \leq \int_\Omega \left| \frac{d}{dx_i} K_2(x,y)u^2(y) \right| \leq C \int_\Omega u^2(y) dy = C'|u|_2^2.$$
	This implies that $\partial_i(\cK_2[ u^2]) \in L^2_{loc}(\Omega)$ and putting $K$ back together we conclude that $\partial_i(\cK [u^2]) \in L^2_{loc}(\Omega) $.
\end{proof}
As a consequence of \eqref{eq:boundconv} we have
\begin{corollary}
	For all $u \in X$, we have the following estimate
	$$ |(\cK [u^2])(x)| \leq [C + \log(1+|x|)]\| u\|_X^2, $$
	for some constant $C>0$.
\end{corollary}
We now move on to the full problem
\begin{equation}\label{eq:pde}
	\begin{cases}
	\begin{alignedat}{3}
	\Delta w(x)            &= u^2(x)                 &\quad  &\text{in } \Omega,\\
	w(\cdot,-\ell)            &= w(\cdot,\ell)             &\quad  &\text{in } \R^2,\\
	\partial_3 w(\cdot,-\ell)            &= \partial_3 w(\cdot,\ell)             &\quad  &\text{in } \R^2,
	\end{alignedat}
	\end{cases}
\end{equation}
First, we define the periodic Schwartz space
\begin{align*}
	\mathcal{S}_p(\Omega) &:= \left\{ \varphi \in C^\infty(\R^3) : \varphi(\cdot,x_3) = \varphi(\cdot,x_3 + \ell) \, \forall x_3 \in \R \right.\\ 
	&\quad \left. \sup_{x \in \Omega} |x^\alpha||\partial_\beta \varphi(x)| < \infty, \, \alpha \in \mathbb{Z}_2^m, \beta \in \mathbb{Z}_3^n, m,n \in \mathbb{N} \right\},
\end{align*}
here we use the following multiindex notation: for $x \in \R^3$ and $\alpha = (\alpha_1, \ldots, \alpha_m) \in \mathbb{Z}_2^m$, $\mathbb{Z}_2 = \{1,2\}$, we define $x^\alpha = x_{\alpha_1} \cdots x_{\alpha_m}$. And for $\varphi \in \mathcal{S}_p(\Omega)$ and $\beta = (\beta_1,\ldots,\beta_n) \in \mathbb{Z}_3^n$, $\mathbb{Z}_3 := \{ 1,2,3\}$ we define $\partial_\beta \varphi := \partial_{\beta_1}\cdots \partial_{\beta_n} \varphi$.

We say that $w \in \mathcal{S}_p(\Omega)^*$ is a solution to \eqref{eq:pde} in the sense of distributions if
\begin{align}\label{eq:dist_solution}
	\langle w, \Delta \varphi \rangle_{\mathcal{S}_p(\Omega)} = \int_{\Omega} u^2(y)\varphi(y) dy \qquad \forall \varphi \in \mathcal{S}_p(\Omega).
\end{align}
In particular, if $w:\overline{\Omega} \to \R$ is a measurable function, we say that it is a solution to \eqref{eq:pde} if
$$ \int_\Omega w(y) \Delta \varphi(y) dt = \int_\Omega u^2(y) \varphi(y) dy \qquad \forall \varphi \in \mathcal{S}_p(\Omega). $$
Observe that, if either $w \in C^2(\overline{\Omega})$ or more generally if $w \in W^{2,p}_{loc}(\Omega)$ for some $p \in [1,\infty)$, and it satisfies \eqref{eq:dist_solution}, then, integrating by parts we see that $w$ satisfies the periodic boundary conditions in \eqref{eq:pde} as well; in the classical sense for $w \in C^2(\overline{\Omega})$ and in the sense of traces if $w \in W^{2,p}_{loc}(\Omega)$.
\begin{remark}\label{remark:test}
Let $\varphi \in \mathcal{S}_p(\Omega)$, then we have $\partial_{x_3} \varphi(x',-\ell) = \partial_{x_3}\varphi(x',\ell)$, so that 
$$ \int_{-\ell}^{\ell} \partial_{x_3}^2 \varphi(x',x_3) d x_3 = 0 \quad \forall x' \in \R^2 \quad \text{thus} \quad \int_{\Omega} \partial_{x_3}^2 \varphi(x',x_3) dx = 0,$$
from this it follows that for all $x_3 \in [-\ell,\ell]$
$$ \int_{\R^2} \partial_{x_3}^2 \varphi(x',x_3) dx = 0,$$
which yields $\partial_{x_3}^2 \mathcal{F}_*[\varphi](0,x_3) = \mathcal{F}_*[\partial_{x_3}^2 \varphi](0,x_3) = 0$ for all $x_3 \in [-\ell,\ell]$. Moreover, we have $\cF_*[\varphi] \in S_p(\Omega)$.
\end{remark}

We proceed to the main result of this section
\begin{proof}[Proof of Theorem \ref{thm:wsolution}]
We recall that for all $\nu \in \Omega$, $K(\xi,\nu)$ has the distributional Fourier transform given by
$$ \mathcal{F}_*[K(\cdot,\nu)](\xi) = - e^{-i \nu' \xi'} \frac{e^{|\xi'||\xi_3-\nu_3|} + e^{|\xi'|(2\ell-|\xi_3-\nu_3|)}}{4 \pi |\xi'|(e^{2\ell|\xi'|}-1)}, $$
in the sense of \eqref{eq:distributiondefinition}.
Now, take a test function $\varphi \in S_p(\Omega)$ and write $\varphi =  \cF_*[\phi]$ for some $\phi \in S_p(\Omega)$ with $\partial_{3}^2 \phi(0,\xi_3) = 0$ for all $\xi_3 \in [-\ell,\ell]$, as explained in Remark \ref{remark:test}. In this way, using the properties of the Fourier transform, we have
\begin{align*}
	\langle \Delta w, \varphi \rangle_{S_p(\Omega)} = \langle \Delta w, \cF_*[\phi] \rangle_{S_p(\Omega)} = \langle \partial_{3}^2 \cF_*[w](\xi) - |\xi'|^2 \cF_*[w](\xi), \phi(\xi) \rangle_{S_p(\Omega)}.
\end{align*}
We will work with the right-hand side of this equation. We start computing
\begin{align}
	\begin{split}\label{eq:eq1}
&-\langle |\xi'|^2\cF_*[w](\xi) , \phi(\xi) \rangle_{\mathcal{S}(\Omega)} = -\langle w(\xi),\mathcal{F}_*[|\cdot|^2\phi(\cdot)](\xi) \rangle_{\mathcal{S}(\Omega)} \\
&\quad =- \int_\Omega\int_\Omega  K(\xi,\nu)u^2(\nu)\mathcal{F}_*[|\cdot|^2\phi(\cdot)](\xi) d\xi d\nu.
\end{split}
\end{align}
Now, from $|x'|\phi(0)=0$ and \eqref{eq:distributiondefinition} we see that 
\begin{align*}
	\int_{\Omega} K(\xi,\nu) \mathcal{F}_*[|\cdot|^2\phi(\cdot)](\xi) d\xi &= -\int_\Omega e^{-i \nu' \xi'} \frac{e^{|\xi'||\xi_3-\nu_3|} + e^{|\xi'|(2\ell-|\xi_3-\nu_3|)}}{4 \pi |\xi'|(e^{2\ell|\xi'|}-1)} |\xi'|^2 \phi(\xi)d\xi \\
	&= \int_{\R^2} e^{-i \nu' \xi'} |\xi'|^2 \mathcal{H}_{|\xi'|}[\phi(\xi',\cdot)](\nu_3) \,d\xi',
\end{align*}
where $\mathcal{H}_{|\xi'|}$ is defined in \eqref{eq:operatorK}. With the aid of Lemma \ref{lemma:greenode} we may rewrite the right-hand side as follows
\begin{align*}
	&\int_{\R^2} e^{-i \nu' \xi'} |\xi'|^2 \mathcal{H}_{|\xi'|}[\phi(\xi',\cdot)](\nu_3) \, d\xi' = \int_{\R^2} e^{-i \nu' \xi'} \left[\partial_{\nu_3}^2 \mathcal{H}_{|\xi'|}[\phi(\xi',\cdot)](\nu_3) - \frac{\phi(\xi',\nu_3)}{2\pi}\right] \, d\xi'\\
&\quad =-\int_{\R^2} e^{-i \xi' \nu'}   \partial_{\nu_3}^2 \left( \int_{-\ell}^\ell \frac{e^{|\xi'||\nu_3-\xi_3|} + e^{|\xi'|(2 \ell -|\nu_3-\xi_3|)}}{4 \pi |\xi'|(e^{2 \ell |\xi'|}-1)} \phi(\xi)d\xi_3\right) d\xi' - \cF_*[\phi](\nu) \\
&\quad=- \int_{\R^2} e^{-i \xi' \nu'}  \partial_{\nu_3}^2 \left( \int_{-\ell}^\ell \frac{e^{|\xi'||\xi_3|} + e^{|\xi'|(2\ell -|\xi_3|)}}{4 \pi |\xi'|(e^{2 \ell |\xi'|}-1)} \phi(\xi',\nu_3-\xi_3)d\xi_3\right)d\xi' -  \cF_*[\phi](\nu)\\
&\quad = -\int_{\R^2}\int_{-\ell}^\ell e^{-i \xi' \nu'}  \frac{e^{|\xi'||\xi_3|} + e^{|\xi'|(2 \ell -|\xi_3|)}}{4 \pi |\xi'|(e^{2 \ell |\xi'|}-1)} \partial_{3}^2 \phi(\xi,\nu_3-\xi_3)d\xi_3 d\xi' - \cF_*[\phi](\nu) \\
&\quad =  -\int_{\R^2}\int_{-\ell}^\ell e^{-i \xi' \nu'}  \frac{e^{|\xi'||\nu_3-\xi_3|} + e^{|\xi'|(2 \ell -|\nu_3-\xi_3|)}}{4 \pi |\xi'|(e^{2 \ell |\xi'|}-1)} \partial_{3}^2 \phi(\xi)d\xi_3 d\xi' - \cF_*[\phi](\nu) 
\end{align*}

Here we observe that, since $\partial_{3}^2 \phi(0,\xi_3) = 0$, the map $e^{-i \xi' \nu'}  \frac{e^{|\xi'||\nu_3-\xi_3|} + e^{|\xi'|(2 \ell -|\nu_3-\xi_3|)}}{4 \pi |\xi'|(e^{2 \ell |\xi'|}-1)} \partial_{3}^2 \phi(\xi)$ is integrable in $\Omega$, so we may switch the integration order to obtain
$$ -\int_{\R^2}\int_{-\ell}^\ell e^{-i \xi' \nu'}  \frac{e^{|\xi'||\nu_3-\xi_3|} + e^{|\xi'|(2 \ell -|\nu_3-\xi_3|)}}{4 \pi |\xi'|(e^{2 \ell |\xi'|}-1)} \partial_{3}^2 \phi(\xi)d\xi_3 d\xi' =  \int_{\Omega} K(\xi, \nu) \mathcal{F}_*[\partial_{3}^2 \phi] (\xi) d \xi. $$
We plug this back in \eqref{eq:eq1}, which shows that
\begin{align*}
&-\langle |\xi'|^2\cF_*[w](\xi) , \phi(\xi) \rangle_{S_p(\Omega)} = -\int_\Omega\int_\Omega K(\xi, \nu)u^2(\nu)\mathcal{F}_*[\partial_{3}^2\phi](\xi) d\xi d\nu + \int_\Omega u^2(\nu) \cF_*[\phi](\nu) d\nu \\
&\quad = -\langle w,\mathcal{F}_*[\partial_3^2\phi] \rangle_{S_p(\Omega)} + \langle u^2,\cF_*[\phi] \rangle_{S_p(\Omega)}
\end{align*}
Thus we have
$$  \langle \partial_3^2 \cF_*[w](\xi) -|\xi'|^2 \cF_*[w](\xi) ,\phi \rangle_{S_p(\Omega)}  = \langle  u^2(\xi), \cF_*[\phi] \rangle_{S_p(\Omega)} = \langle  u^2(\xi), \varphi \rangle_{S_p(\Omega)},$$
which is equivalent to
$$ \langle \Delta w,\varphi \rangle_{S_p(\Omega)}  = \langle u^2,\varphi \rangle_{S_p(\Omega)},$$
as we observed at the start of the proof. This shows that $w$ is a solution to \eqref{eq:pde} in the sense of distributions.

Now, since $\cK [u^2] \in W^{1,2}_{loc}(\Omega)$ it follows that $w$ solves the equation $ \Delta w = u^2$ in the weak sense of $ W^{1,2}_{loc}(\Omega)$. Here, we observe that $u^2 \in W^{1,3/2}(\Omega)$, since from the Sobolev embedding $u \in L^s(\Omega)$ with $s \in [2,6]$, then $\partial_i(u^2) = 2 u \partial_i u \in L^{3/2}(\Omega)$ using the Hölder inequality for $u \in L^6(\Omega)$ and $\partial_i u \in L^2(\Omega)$. Elliptic regularity theory implies that $w \in W^{3,3/2}_{loc}(\Omega)$ and finally, we use again the the Sobolev embedding $W^{3,3/2}_{loc}(\Omega) \hookrightarrow C^{0,\alpha}_{loc}(\Omega)$ for all $\alpha \in (0,1)$.

Now, we conclude by showing that the periodic extension $\overline{w}$ is $C^{0,\alpha}_{loc}$ in the domain $\R^2 \times [0,2\ell]$, which covers the entire space $\R^3$ by periodicity. Indeed, we repeat the argument above for the function $u_\tau := T_\tau u \in X$ where $\tau := (0,-\ell)$ and $T_\tau$ is the translation operator in \eqref{eq:translation_op}. By the argument above, we see that $\cK [u_\tau^2] \in C^{0,\alpha}_{loc}(\Omega)$. Now, we we recall from Theorem \ref{thm:asymtoticsGreen} part 1, that the periodical extension $\overline{K} : \R^6_* \to \R$ satisfies $\overline{K} \in C^\infty(\R^3_*)$. Then, we derive the following
$$ \overline{w}(x-\tau) = \int_\Omega \overline{K}(x-\tau,y) u^2(y) dy = \int_\Omega \overline{K}(x,y+\tau) u^2(y) dy = \int_\Omega K(x,y)u_\tau^2(y) dy,$$
and in the second equality we used Corollary \ref{corollary:translation_K}. Thus we have shown, that $T_\tau w = \cK [u_\tau^2]$ thus $\overline{w}$ has the desired property in $\R^2 \times [0,2\ell]$. 
\end{proof}
As a corollary we obtain Item 3 in Theorem \ref{thm:asymtoticsGreen}
\begin{proof}[Proof of Theorem \ref{thm:asymtoticsGreen} continued (Item 2)]
	Now, we conclude the proof of Theorem \ref{thm:asymtoticsGreen}. Let $f \in C^\infty_{c}(\Omega)$. Following the proofs of Lemma \ref{lemma:convolution} and Theorem \ref{thm:wsolution} to $f$ in place of $u^2$ we see that $w := \cK[f]$ is a distributional solution to
	    	\begin{equation*}
		\begin{cases}
			\begin{alignedat}{3}
				\Delta w(x)            &= f(x)             &\quad  &\text{in } \Omega,\\
				w(\cdot,-\ell)            &= w(\cdot,\ell)             &\quad  &\text{in } \R^2,\\
				\partial_{3}w(\cdot,-\ell)&= \partial_{3}w(\cdot,\ell) &\quad  &\text{in } \R^2.
			\end{alignedat}
		\end{cases}
	\end{equation*}
	Then, elliptic regularity theory implies that $w \in C^\infty(\Omega)$ and we may extend the regularity to the periodical extension $\overline{w}:\R^3 \to \R$ as in Theorem \ref{thm:wsolution}. This concludes the proof.
\end{proof}

We now make some observations regarding the relationship of $K$ with the Green functions of the 2- and 3-dimensional Poisson equation. We start with the following 

\begin{theorem}\label{thm:green2d}
	Let $u \in X$ be such that $\partial_{3}u \equiv 0$. Then for almost all $x \in \Omega$ it holds
	$$ \cK[ u^2] (x) = \frac{1}{2 \ell} \int_{\Omega} \log|x'-y'|u^2(y)dy$$
\end{theorem}
\begin{proof}
	Let $\varphi \in C^{\infty}_c(\Omega)$, and set $\varphi(x) = \cF_*[\phi](x)$, where $\phi \in S_p(\Omega)$. Then from Lemma \ref{lemma:greenode}
	\begin{align*}
		&\langle \cK[ u^2], \varphi \rangle_{S_p(\Omega)} = \langle \cK [u^2], \cF_*[\phi] \rangle_{S_p(\Omega)} = \int_{\Omega} \int_{\Omega} K(\xi, \nu) u^2(\nu) \cF_*[\phi](\xi) d \xi d\nu \\
		&\quad = -\int_{\Omega} \int_{\Omega} \frac{e^{-i \xi' \nu'} (e^{|\xi'||\xi_3 - \nu_3|} + e^{|\xi'|(2\ell-|\xi_3-\nu_3|)})}{4 \pi |\xi'|(e^{2\ell|\xi'|}-1)}u^2(\nu)(\phi(\xi)-\phi(0,\xi_3)1_{|\xi'|<m}(\xi)) d\xi d\nu \\
		&\quad = \int_{\Omega} \int_{\R^2}(\phi(\xi)-\phi(0,\xi_3)1_{|\xi'|<m}(\xi))|\xi'|^2e^{-i\nu' \xi'} \mathcal{H}_{|\xi'|}\left[ \frac{u^2(\nu', \cdot)}{|\xi'|^2}\right](\xi_3) d\nu' d\xi \\
		&\quad = - \int_{\Omega} \int_{\R^2} (\phi(\xi)-\phi(0,\xi_3)1_{|\xi'|<m}(\xi)) e^{-i \xi' \nu'} \frac{u^2(\nu',\xi_3)}{2\pi |\xi'|^2} d\nu' d\xi\\
		&\quad = -\int_{\R^2} \int_{-\ell}^{\ell}\int_{\R^2} e^{-i \xi' \nu'} \frac{\phi(\xi)-\phi(0,\xi_3)1_{|\xi'|<m}(\xi)) }{2\pi |\xi'|^2} u^2(\nu',\xi_3) d\xi' d\xi_3 d\nu' \\
		&\quad =  \int_{\R^2} \int_{-\ell}^{\ell} \int_{\R^2} \log|\xi'-\nu'| \cF_*[\phi](\xi) u^2(\nu',\xi_3) d\xi' d\xi_3 d\nu' = \int_{\R^2} \int_{\Omega} \log|\xi'-\nu'| \cF_*[\phi](\xi) u^2(\nu',\xi_3) d\xi d\nu' \\
		&\quad = \int_{\Omega} \int_{\R^2} \log|\xi'-\nu'|\varphi(\xi) u^2(\nu',\xi_3)  d\nu' d\xi = \frac{1}{2 \ell} \int_\Omega \int_\Omega  \log|\xi'-\nu'|\varphi(\xi) u^2(\nu',\xi_3)  d\nu d\xi
	\end{align*}
Since $\varphi \in C^\infty_{c}$ is arbitrary, the identity holds a.e. in $\Omega$.
\end{proof}
Thus the Green function $K$ collapses to the 2-dimensional fundamental solution to the Poisson equation when convoluted with a function independent of $x_3$. Now, we prove that, as $\ell \to \infty$, it converges in an appropriate sense to the $3$-dimensional fundamental solution. 

To make emphasis on the dependence of $K$ on $\ell$, we will write $K(x,y;\ell)$ to denote the Green function from Theorem \ref{thm:asymtoticsGreen} in the domain $\Omega_\ell := \R^2 \times [-\ell,\ell]$ and $ \cK_\ell $ for the corresponding operator
$$ \cK_{\ell}[ f] (x) := \int_\Omega K(x,y;\ell) f(y) dy \in [-\infty,\infty].$$

\begin{theorem}\label{thm:Green3dim}
	Let $\varphi \in C^{\infty}_{c}(\R^3)$. Then we have
	$$ \int_{\Omega_\ell} \int_{\Omega_\ell} K(x,y;\ell)\varphi^2(y) \varphi^2(x) dy dx \to - \frac{1}{4 \pi} \int_{\R^3}\int_{\R^3} \frac{\varphi^2(y) \varphi^2(x)}{|x-y|}dy  \quad \text{as } \ell \to \infty.$$
\end{theorem}
\begin{proof}
Let $x \in \R^3$, then, for $\ell$ large enough we have $x \in \Omega_\ell$ and $\text{supp }\varphi \subset \Omega_\ell$. Then $\varphi^2 = \cF_*[\phi]$ for some $\phi \in S(\R^3)$ with $\text{supp } \phi \subset \Omega_\ell$. Now, we have the following
\begin{align}\label{eq:Gconvphi}
\begin{split}
	&\cK_\ell [\varphi^2](x) = \cK_\ell[ \cF_*[\phi]] (x)  = \int_{\Omega_\ell} K(x,y;\ell) \cF_*[\phi](y) dy \\ 
	&\quad = -\int_{-\ell}^{\ell} \int_{\R^2} \frac{ e^{-i x' y'} (e^{|y'||y_3-x_3|} +e^{|y'|(2\ell -|y_3- x_3|)}) }{4 \pi |y'|(e^{2\ell |y'|}-1)} (\phi(y) - \phi(0,y_3)1_{\{|y'|<m\}}(y)) d y' dy_3\\ 
	&\quad = -\int_{\R^3}  \frac{ e^{-i x' y'} (e^{|y'||y_3-x_3|} +e^{|y'|(2\ell -|y_3- x_3|)}) }{4 \pi |y'|(e^{2\ell |y'|}-1)}  (\phi(y) - \phi(0,y_3)1_{\{|y'|<m\}}(y)) d y.
\end{split}
\end{align}
Then, since $\cK_\ell[\varphi^2 ]\in L^\infty_{loc}(\R^3)$ we can multiply \eqref{eq:Gconvphi} by $\varphi^2$ and integrate over $\R^3$ to obtain
\begin{align}\label{eq:doubleintegralG}
\begin{split}
		&\int_{\R^3} \cK_\ell [\varphi^2](x) \varphi^2(x) dx  = \int_{\R^3} \int_{\R^3} K(x,y;\ell)\varphi^2(x) \varphi^2(y) dx dy \\
	&\quad =- \int_{\R^3} \int_{\R^3}  \frac{ e^{-i x' y'} (e^{|y'||y_3-x_3|} +e^{|y'|(2\ell -|y_3- x_3|)}) }{4 \pi |y'|(e^{2\ell |y'|}-1)}  (\phi(y) - \phi(0,y_3)1_{\{|y'|<m\}}(y)) \varphi^2(x) dx dy.
\end{split}
\end{align}
Now, there is $k \in \mathbb{N}$ such that $\text{supp}(\phi) \subset \R^2 \times [-k,k] \subset \Omega_\ell$. We also assume that $k \geq 1$ and fix $c >1$ such that $e^{2k} \geq c/(c-1)$. In this way we see that for $|y'| >1$ and $\ell > k$ it holds $ \frac{c}{c-1} \leq e^{2k} \leq e^{2\ell|y'|} $, which after rearranging yields
$$ \frac{e^{2 \ell |y'|}}{e^{2 \ell |y'|}-1} \leq c.$$
Then, let $\chi$ be the indicator function of $\{x \in \R^3 : |x'|< m, \, |x_3| < k\}$. In this way we have the following estimate
\begin{align*}
	&\left|\varphi^2(x) \frac{ e^{-i x' y'} (e^{|y'||y_3-x_3|} +e^{|y'|(2\ell -|y_3- x_3|)}) }{4 \pi |y'|(e^{2\ell |y'|}-1)} \left( \phi(y) - \phi(0,y_3)1_{\{|y'|<m\}}(y) \right) \right| \\
&\quad \leq \varphi^2(x) \left( \frac{ e^{|y'||y_3-x_3|} +e^{|y'|(2\ell -|y_3- x_3|)} }{4 \pi |e^{2\ell |y'|}-1|}  |\partial_{x'}\phi|_{\infty} \chi(y) +   \frac{ e^{|y'||y_3-x_3|} +e^{|y'|(2 \ell -|y_3- x_3|)} }{4 m \pi |e^{2\ell |y'|}-1|}(1-\chi(y))\phi(y)\right) \\
&\quad \leq \varphi^2(x) \frac{ e^{2\ell|y'|} }{2 \pi |e^{2\ell |y'|}-1|} \left(|\partial_{x'}\phi|_{\infty} \chi(y)  +  \frac{ \phi(y) }{ m} \right) \\
&\quad \leq \frac{c \varphi^2(x)}{2\pi}  \left(|\partial_{x'}\phi|_{\infty} \chi(y) + \frac{\phi(y)}{m} \right)  \in L^1(\R^3 \times \R^3).
\end{align*}
Then we may apply the dominated convergence theorem in \eqref{eq:doubleintegralG} together with \eqref{eq:hankel} to obtain
\begin{align*}
	\lim_{\ell \to \infty} \int_{\R^3} \int_{\R^3}  K(x,y;\ell)\varphi^2(x) \varphi^2(y) dx dy &= -\frac{1}{4 \pi} \int_{\R^3} \int_{\R^3} \frac{e^{-i x' y'} e^{-|y'||y_3 - x_3|}}{|y'|}\phi(y) \varphi^2(x)  dx dy \\
	&\quad = - \frac{1}{4 \pi} \int_{\R^3} \int_{\R} \int_{\R^2} \mathcal{F}_*\left( \frac{1}{|\cdot - x|} \right)(y) \phi(y) \varphi^2(x) dy' dy_3 dx\\
	&\quad = -\frac{1}{4 \pi }\int_{\R^3} \int_{\R^3} \frac{\varphi^2(y) \varphi^2(x)}{|y - x|} d y dx.
\end{align*}
\end{proof}

\section{The variational setting for the Choquard equation}\label{sec:variational_setting}
In this section, we set up our variational framework and prove some initial technical results. For the initial part of this section, we only assume $a \in L^{\infty}(\Omega)$ and $\text{ess inf}_{\R^2}a > 0$. Later in this section we will introduce further symmetry conditions. 

In the space $H^1_{p}(\Omega)$ we consider the norm
$$ \|u\|_a^2 := \int_{\Omega} a(x)u^2 + |\nabla u|^2 dx, $$
which is equivalent to the usual Sobolev norm in this space. Then, we introduce the subspace
$$ X := \{u \in H^1_{p}(\Omega) : |u|_* < \infty \} \quad \text{where} \quad \|u\|_X^2 := \|u\|_{a}^2 + |u|_*^2. $$
Then we recall the decomposition of $K$ from Theorem \ref{thm:asymtoticsGreen}: we write $K = K_1 + K_2$, where \mbox{$K_1(x,y) := -(4\pi |x-y|)^{-1}$} and $K_2 \in C^{\infty}(\overline{\Omega}^2)$ satisfies
\begin{align}\label{eq:Gasymptotic}
 \frac{1}{C_K} \leq  \frac{K_2(x,y)}{\log(1+|x-y|)} \leq C_K \quad \text{ for } |x-y| \geq R,
\end{align}
for some constants $C_K > 0$, $R>1$. In particular we see that $K(x,y) > 0$ if $|x-y| \geq R$.

In the same spirit as the variational setting in \cite{cw} we consider the bilinear forms
\begin{align*}
	B_i(u,v) := \int_{\Omega} \int_{\Omega} K_i{}(x,y)u(x)v(y) dx dy  \quad \text{for } i=1,2, \qquad B_0(u,v) := B_1(u,v) + B_2(u,v),
\end{align*}
whenever they are defined.

Now, if $u,v \in L^2(\Omega)$, from Hölder's inequality and the Hardy-Littlewood-Sobolev inequality \cite{mvs2} we see that
\begin{align}\label{eq:B1ineq}
	\begin{split}
	|B_1(u,v)| &= \frac{1}{4 \pi} \int_{\Omega} \int_{\Omega} \frac{|u(x) v(y)|}{|x-y|} dx dy  \leq C |u|_{6/5} |v|_{6/5}
	\end{split}
\end{align}
for some constant $C>0$.
On the other hand, using the Cauchy-Schwarz inequality
\begin{align}\label{eq:B2ineq}
	\begin{split}
	&\left| \int_{\Omega} \int_{\Omega} \frac{K_2(x,y)}{\log(1+|x-y|)} \log(1+|x-y|)u(x)v(x)w(y)z(y) 1_{\{|x-y|>R\}}(x,y) \, dxdy \right| \\
	&\quad \leq C_K \int_{\Omega} \int_{\Omega} \log(1+|x-y|)|u(x)v(x)w(y)z(y)|dx dy \\
	&\quad \leq C_K \int_{\Omega} \int_{\Omega} \left[ \log(1+|x|) + \log(1+|y|) \right] |u(x)v(x)w(y)z(y)|dx dy \\
	&\quad \leq C_K ( |u|_*|v|_*|w|_2|z|_2 + |u|_2|v|_2|w|_*|z|_*).
	\end{split}
\end{align}
Further, with a second application of the Cauchy-Schwarz inequality
\begin{align}\label{eq:B2ineq2}
	\begin{split}
		&\left| \int_\Omega \int_\Omega K_2(x,y)1_{|x-y|<R} u(x)v(x)w(y)z(y)dx dy \right| \\
		&\quad \leq \sup_{|x-y|<R}|K_2(x,y)| \int_\Omega \int_\Omega |u(x)v(x)w(y)z(y)| dx dy \\
		&\quad  \leq C |u|_2|v|_2|w|_2|z|_2
	\end{split}
\end{align}

So, putting \eqref{eq:B1ineq}, \eqref{eq:B2ineq} and \eqref{eq:B2ineq2} together, we see that, there is $C>0$ such that for all $u,v \in X$
\begin{equation}\label{eq:estimatesGi}
	|B_1(u^2,v^2)| \leq C |u|_{12/5}^2 |v|_{12/5}^2 \qquad  |B_2(u^2,v^2)| \leq C (|u|_*^2|v|_2^2 + |u|_2^2|v|_*^2),
\end{equation}
and we set $V_i(u) := B_i(u^2,u^2)$. 

Now, observe, from \eqref{eq:Gasymptotic}, that
\begin{align}\label{eq:B2_lower_bound}
	\begin{split}
	B_2(u^2,v^2) &= \left( \int\displaylimits_{\substack{x,y \in \Omega \\ |x-y|<R}} + \int\displaylimits_{\substack{x,y \in \Omega \\ |x-y|>R}} \right)K_2(x,y)u^2(x)v^2(y) dx dy\\
	&\geq  - \sup\displaylimits_{\substack{x,y \in \Omega \\ |x-y|<R}} |K(x,y)| |u|_2^2 |v|_2^2 + C_K^{-1} \int\displaylimits_{\substack{x,y \in \Omega \\ |x-y|>R}}\log|x-y| u^2(x)v^2(y) dx dy \\
	& \geq -C |u|_2^2|v|_2^2 \quad \forall u,v \in X,
	\end{split}
\end{align}
where we use the fact that $K(x,y) = K(x-y,0)$ from Corollary \ref{corollary:translation_K} and $K \in C^\infty(\overline{ \Omega }^2 )$ to bound $K(x,y)$ in $\{x,y \in \Omega : |x-y| < R\}$.

We have the following property
\begin{lemma}\label{lemma:B2lemma}
	Let $(u_n)_n$ be a sequence in $L^2(\Omega)$ such that $u_n \rightarrow u \in L^2(\Omega) \setminus \{0\}$ pointwise a.e. on $\Omega$. Moreover, let $(v_n)_n$ be a bounded sequence in $L^2(\Omega)$ such that
	\[
	\sup_{n\in\mathbb{N}} B_2(u_n^2, v_n^2) < \infty.
	\]
	Then, there exists $n_0 \in \mathbb{N}$ and $C > 0$ such that $|v_n|_* < C$ for $n \geq n_0$.
	
	If, moreover,
	\[
	B_2( u_n^2, v_n^2) \rightarrow 0 \quad \text{and} \quad |v_n|_2^2 \rightarrow 0 \quad \text{as} \quad n \rightarrow \infty,
	\]
	then
	\[
	|v_n|_* \rightarrow 0 \quad \text{as} \quad n \rightarrow \infty, \quad n \geq n_0.
	\]
\end{lemma}
\begin{proof}
	Since $u \neq 0$, there is $M>R>0$ such that $u|_{U_{M}} \neq 0$ \footnote{Here $U_M = U_M(0):=\{x \in \Omega : |x| < M \}$.}. Then, by Egorov's theorem, we may find $A \subset U_M$ such that $|u_n(x)| > \delta > 0$ for all $x \in A$. We also observe that if $x \in U_M$ and $y \in \Omega \setminus U_{2M}$ then
	$$ 1 + |x-y| \geq 1 + |y| -|x| \geq 1 + |y| - M \geq \frac{|y|}{2}+1 > \sqrt{1+|y|}.$$
	Further, since $R < M$ we have $|x-y|>R$ for $ x\in A$ and $y \in \Omega \setminus U_{2M}$. In this way, using the inequality above and \eqref{eq:Gasymptotic}, we have
	\begin{align*}
		B_2(u_n^2,v_n^2) &\geq \int_{\Omega \setminus U_{2M}} \int_{A} K_2(x,y)u_n^2(x)v_n^2(y)dxdy \\
		&= \int_{\Omega \setminus U_{2M}} \int_{A} \frac{K_2(x,y)}{\log(1+|x-y|)} \log(1+|x-y|)1_{|x-y|>R} u_n^2(x) v_n^2(y) dx dy, \\
		&\geq \frac{|A|\delta^2}{2C_K} \int_{\Omega \setminus U_{2M}}\log(1 + |y|)v_n^2(y)dy = \frac{|A| \delta^2}{2C_K} (|v_n|_*^2 - \log(1+2M)|v_n|_2^2).
	\end{align*}
	To conclude, we observe that the sequences $B_2(u_n^2,v_n^2)$ and $|v_n|_2^2$ are bounded. Thereby $(u_n)_n$ is bounded in $X$ as well. The second claim follows in a similar way if we assume that both $B_2(u_n^2,v_n^2)$ and $|v_n|_2$ converge to 0.
\end{proof}

Now, we introduce the functional $\Phi:X \to \R$
$$ \Phi(u) := \frac{\|u\|_a^2}{2} + \frac{V_0(u)}{4}, $$
which is well-defined in $X$. The following properties of the functionals defined above can be checked in an analogous way as in \cite{cw}
\begin{lemma}\label{lemma:propertiesI}
\begin{itemize}
	\item[(i)] The space $X$ is compactly embedded in $L^s(\Omega)$ for all $s \in [2, 6)$.
	
	\item[(ii)] The functionals $V_0$, $V_1$, $V_2$, and $\Phi$ are of class $C^2$ on $X$. Moreover, $V_i(u, v) = 4B_i(u^2, uv)$ for $u, v \in X$ and $i = 0, 1, 2$.
	
	\item[(iii)] $V_1$ is of class $C^1$ on $L^4(\Omega)$.
	
	\item[(iv)] $V_2$ is weakly lower semicontinuous on $H^1(\Omega)$, as a consequence $V_0$ is also weakly lower semicontinuous on $H^1(\Omega)$.
	
	\item[(v)] $\Phi$ is weakly lower semicontinuous on $X$.
	
	\item[(vi)] $\Phi$ is lower semicontinuous on $H^1(\Omega)$.
\end{itemize}
\end{lemma}

Now, we say that $u \in X$ is a \textit{weak solution} to the equation
\begin{equation}\label{eq:choquard}
	- \Delta u + a(x) u + (\cK[u^2])u = 0 \quad \text{in } \Omega
\end{equation}
if the following identity holds
\begin{equation}
	\int_\Omega a(x)u(x) \varphi(x) + \nabla u(x) \cdot \nabla \varphi(x) dx + \int_{\Omega} \int_{\Omega} K(x,y)u^2(y)u(y)\varphi(x) dx dy = 0,
\end{equation}
for all $\varphi \in X$. Then, weak solutions to \eqref{eq:choquard} correspond to critical points of $\Phi$.

Moreover, can relate the weak solutions to \eqref{eq:choquard} and the solutions to the Schrödinger-Poisson system
\begin{align}\label{eq:SPsystem}
	\begin{cases}
		- \Delta u + a(x) u + u w = 0 &\text{in } \R^3,\\
		\Delta w = u^2 &\text{in } \R^3.
	\end{cases}
\end{align}
as follows
\begin{theorem}\label{thm:solsystem}
	Let the periodic extension of $a$ satisfy $\overline{a} \in C^{0,\alpha}_{loc}(\R^3)$. Further, let $u \in X$ be a weak solution to \eqref{eq:choquard} and define $w:= \cK [ u^2]$. Then the periodic extension $(\overline{u},\overline{w})$ is a periodic solution to the system \eqref{eq:SPsystem}, with $\overline{u} \in C^{2,\alpha}_{loc}(\R^3)$ and $\overline{w} \in C^{4,\alpha}_{loc}(\R^3)$. Moreover, $w$ satisfies
	$$ w(x) \sim \log(1+|x|) \quad \text{as } |x| \to \infty, x \in \Omega. $$
\end{theorem}
\begin{proof}
	The asymptotic behavior of $w$ at infinity, as well as the $w \in W^{1,2}_{loc}(\Omega)$ was already derived in Lemma \ref{lemma:convolution}. We proceed with the properties of the pair $(u,w)$. Indeed, we recall from Theorem \eqref{thm:wsolution}, that $w$ solves the equation $\Delta w = u^2$ in $W_{loc}^{1,2}(\Omega)$ and $\overline{w} \in C^{0,\alpha}_{loc}(\R^3)$. Going back to the first equation in \eqref{eq:SPsystem} we have
	$$ - \Delta u + (a(x) + w(x))u = 0 \quad \text{in } \Omega, $$
	where $a + w \in C^{0,\alpha}_{loc}(\Omega)$. Now, it follows from elliptic regularity theory that $u \in C^{2,\alpha}_{loc}(\Omega)$. Then, we return back to the linear equation $\Delta w = u^2$, where again from elliptic regularity theory we conclude that $w \in C^{4,\alpha}_{loc}(\Omega)$. Now, the fact that $u \in H^1_{p}(\Omega)$, allows us to extend the regularity to $\R^3$, by considering the periodic extensions of each function to $\R^3$ as we did in the proof of Theorem \ref{thm:wsolution}.
\end{proof}

We have the following bounds

\begin{lemma}\label{lemma:Imountaingeometry}
	There is $\beta > 0$ such that, for all $0 < r \leq \beta$ the inequalities below hold
\begin{align*}
	\inf_{\|u\|_a = r} \Phi(u) &> 0, \\
	\inf_{\|u\|_a = r} \Phi'(u)u &> 0.
\end{align*}
\end{lemma}
\begin{proof}
Using the estimates \eqref{eq:estimatesGi}, \eqref{eq:B2_lower_bound} and the Sobolev embedding we derive
\begin{align*}
	\Phi(u) &= \frac{\|u\|_a^2}{2} + \frac{V_0(u)}{4} \geq \frac{\|u\|_a^2}{2} + \frac{V_1(u)}{4} - C|u|_2^4 \geq \frac{\|u\|_a^2}{2} - C'|u|_{12/5}^4 - C|u|_2^4\\
	&\geq C'' (\frac{\|u\|_a^2}{2} - \frac{\|u\|_a^4}{4}),
\end{align*}
for some constants $C, C',C'' >0$. The first inequality then follows taking $\beta > 0$ small enough. The second inequality can be obtained in a similar way.
\end{proof}

We define the Nehari manifold
$$ \mathcal{N} := \{u \in X \setminus \{0\} : \Phi'(u)u = 0\} = \{ u \in X \setminus \{0\} : \|u\|_a^2 = -V_0(u) \} $$
which can be equivalently defined through the \textit{fiber maps} $f_u : \R \to \R$ given by 
$$f_u(t) := \Phi(tu) = \frac{\|u\|_a^2}{2}t^2 + \frac{V_0(u) }{4}t^4 \qquad u \in X.$$
In this way, it easy to see that 
$$ \mathcal{N} = \{ u \in X : f_u'(1) = 0\}. $$
Now, we enlist some properties of the Nehari manifold
\begin{lemma}\label{lemma:Nehariproperties}
	\begin{itemize}
		\item[(i)] The set $\mathcal{N}$ is a $C^2$-manifold and it is a natural constraint for $\Phi$, i.e. critical points of the restriction $\Phi|_{\mathcal{N}}$ are critical points of $\Phi:X \to \R$.
		\item[(ii)] $\mathcal{N}$ is nonempty, closed and bounded away from $0$. Moreover, $ \inf_{\mathcal{N}} \Phi > 0$.
		\item[(iii)] $u \in \mathcal{N}$ if and only if $t=1$ is global maximum of $f_u$.
		\item [(iv)] Let $V_0(u) < 0$, then $t u \in \mathcal{N}$ if and only if $t^2 = t_u^2 := -\|u\|_a^2/V_0(u)$.
	\end{itemize}
\end{lemma}
\begin{proof}
Part $(i)$ is a standard result. For $(ii)$, we start by proving that $\mathcal{N}$ is nonempty. We recall the asymptotic description \ref{thm:asymtoticsGreen}, so we may find $\delta >0$ small enough that $K_1(x) < 0$ for all $|x|<\delta$. Then, we take a nonzero test function $\varphi \in C^{\infty}_c(U_{\delta})$ and for this function it holds $ V_0(\varphi) = V_1(\varphi) < 0 < \| \varphi \|_a$. It follows now that the fiber map $t \mapsto \Phi(t\varphi)$ has a unique positive critical point, say at $t_0 > 0$, and then $t_0 \varphi \in \mathcal{N}$.

Now, we can see that $\mathcal{N}$ is bounded away from $0$, using Lemma \eqref{lemma:Imountaingeometry} to conclude that for all $u \in \mathcal{N}$ we have the inequality $\|u\|_a \geq \alpha$. From this fact, we conclude immediately that $\mathcal{N}$ is closed. Part $(iii)$ follows from the observation that if $u \in \mathcal{N}$, then $f_u''(1) = - 2\|u\|_a^2 \leq -2 \alpha < 0$; and the inverse statement is immediate. Finally, $(iv)$ follows from the definition of $\mathcal{N}$, since $tu \in \mathcal{N}$ if and only if $t^2\|u\|_a^2 = \|tu\|^2_a = -V_0(tu) = - t^4 V_0(u)$.
\end{proof}

Now, we introduce some notation related to the symmetrical setting. Let $O(2)$ be the orthogonal group in $\R^2$. We write its elements as $g_{A}$ where $A \in O(2)$ and consider the action of this group on $\R^3$ given by 
$$ g_{A}(x',x_3) = (Ax',x_3) \qquad x = (x',x_3) \in \R^3.$$
In this way, $O(2)$ acts isometrically on $L^p(\Omega)$ by setting
\begin{equation}\label{eq:action_O2}
	(g_A * u)(x)= u(g_A^{-1 } x) \qquad \text{for $u \in L^p(\Omega)$, $A \in O(2)$.}
\end{equation}
Then, we consider the $O(2)$-invariant subspace 
$$ X_r := \{ u \in X : g_A \ast u = u, \, \forall A \in O(2)\}. $$

Now, for the remainder of this paper we will assume that
\begin{equation}\label{eq:symmetries_a}
	a(Ax',x_3) = a(x', x_3) \quad \text{a.e. in } \Omega \text{ for all } A \in O(2).
\end{equation}

We will continue our analysis with the restriction $\Phi:X_r \to \R$, which by \eqref{eq:symmetries_a} is an $O(2)$-invariant functional. We require the following standard result to make sure that we recover critical points of the full functional $\Phi : X \to \R$, see \cite{palais_sym}

\begin{lemma}\label{lemma:symmetriccriticality}
	Let $u \in X$ be nontrivial critical point of $\Phi : {X_r} \to \R$, then $u$ is a nontrivial critical point of $\Phi:X \to \R$.
\end{lemma}

\section{Radial ground states and high energy solutions}\label{multiplicity}
In this section we describe a method to produce high energy solutions to \eqref{eq:choquard}. For the remaining of this section we set $\mathcal{N}_r := \mathcal{N} \cap X_r$ which is a $C^2$ submanifold of $X_0$. Moreover, $\mathcal{N}_r$ inherits a natural Riemannian structure from $X$ by restricting $\| \cdot \|_X$ to the tangent space $T\mathcal{N}_r$. We will denote the canonical dual norm over the cotangent space $T^*\mathcal{N}_r$ by $\| \cdot \|_{X^*}$.

Now, we recall the notion of a \textit{Palais-Smale sequence}. We say that $(u_n)_n \subset \mathcal{N}_r$ is a Palais-Smale sequence at $c \in \R$, $(PS)_c$-sequence for short, if
\begin{equation}\label{eq:ps}
	\Phi(u_n) \to c \qquad \| \Phi|_{\mathcal{N}_r}' \|_{X^*} \to 0.
\end{equation}
We say that $\Phi:\mathcal{N}_r \to \R$ has the \textit{Palais-Smale} property at $c$, $(PS)_c$-property for short, if any $(PS)_c$-sequence contains a subsequence that converges strongly in $X$.

We have the following
\begin{lemma}\label{lemma:H1boundedsequences}
	Let $(u_n)_n \subset X_r$ be bounded in $H^1(\Omega)$. Then, either $u_n \to 0$ in $L^s(\R^2)$ for all $s \in (2,6)$, or there is $u \in H^1(\Omega) \setminus \{0\}$ such that, $g_A \ast u = u$ for all $A \in O(2)$ and, up to a subsequence, we have
	\begin{align*}
		u_n \rightharpoonup u &\quad \text{weakly in } H^1(\Omega), \\
		u_n \to u &\quad \text{in } L^s_{loc}(\Omega) \quad \text{for } s \in (2,6),\\
		u_n \to u &\quad \text{a.e. in } \Omega.
	\end{align*}
\end{lemma}	
\begin{proof}
	Let $x' \in \R^2$ and define $U(x') := \{ z \in \Omega : |x'-z'| < 1 \}$. Then set
	$$ \mu := \liminf_{n \in \mathbb{N}}\sup_{x' \in \R^2}\int_{U(x')} u_n^2(y)dy,$$
	If $\mu = 0$, then it follows from Lions' lemma, see \cite{l}, that $u_n \to 0$ in $L^s(\Omega)$ for all $s \in (2,6)$. Otherwise, $\mu > 0$ and thus there is a sequence of $x_n' \in \R^2$ such that
	\begin{equation}\label{eq:mu}
		\lim_{n \to \infty} \int_{U(x_n')} u_n^2(y) dy \to \mu.
	\end{equation}
	Then we find $u \in H^1(\Omega)$ such that, up to a subsequence, we have
	\begin{align*}
		u_n(\cdot - x_n) \rightharpoonup u &\quad \text{weakly in } H^1(\mathbb{R}^2), \\
		u_n(\cdot - x_n) \to u &\quad \text{in } L^s_{loc}(\mathbb{R}^2) \quad \text{for } s \in (2,6),\\
		u_n(\cdot - x_n) \to u &\quad \text{a.e in } \mathbb{R}^2,
	\end{align*}
	where $x_n := (x_n',0) \in \Omega$. Moreover, from \eqref{eq:mu} and the convergence in $L^s_{loc}$, it is easy to see that $u \neq 0$. It remains to show that we can set $x_n = 0$ for all $n \in \mathbb{N}$. The result is clear if we can show that $(x_n)_n$ is bounded. If this is not the case, passing to a subsequence, we may assume that $|x_n| \to \infty$. Then we let $A_m \in O(2)$ be a rotation of angle $2 \pi /m$ in the plane $\R^2$ and consider the sequences $x_n^j := A_m^j x_n$. Observe that, 
	$$ | x_n^j  - x_n^k | \to \infty \qquad \text{for } k,j = 0,...,m-1, \, k \neq j,$$
	and thus, for $n$ large enough and $j \neq k$, we have
	$$ U(x_n^j) \cap U(x_n^k) = \emptyset.$$
	Since $u \in X_0$ it follows that
	$$ \int_{U(x_n^k)}u_n^2(y) dy = \int_{U(x_n)}u_n^2(y) dy \to \mu \qquad k = 0,\ldots,m-1,$$
	Thus
	\begin{align*}
		|u_n(\cdot - x_n)|_2^2 = |u_n|_2^2 = \int_{\Omega} u_n^2(y) dy \geq \sum_{k=1}^m \int_{B_1(x_n^k)} u_n^2(y) dy = m \mu + o(1),
	\end{align*}
	for all $m \in \mathbb{N}$. This contradicts that $u_n(\cdot - x_n)$ is bounded in $H^1(\Omega)$. Then we may set $x_n = 0$ and from the pointwise convergence we obtain $g_{A} \ast u = u$ for all $A \in O(2)$. 
\end{proof}
Combining this result with Lemma \ref{lemma:B2lemma} we obtain
\begin{corollary}\label{cor:weakconvX}
	Let $(u_n)_n \subset X_r$ be a sequence bounded in $H^1(\Omega)$ and such that $V_2(u_n)$ is bounded. Then, passing to a subsequence, we have either $u_n \to 0$ in $L^s(\Omega)$ for all $s \in (2,6)$, or there is $u \in X_r$ such that $u_n \rightharpoonup u$ weakly in $X$.
\end{corollary}

To prove that $\Phi$ satisfies the $(PS)_c$-property, we will need to borrow the following result from \cite[Lemma 2.6]{cw}, whose proof can be easily adapted to our setting
\begin{lemma}\label{lemma:weakconvergence}
	Let $(u_n)_n$, $(v_n)_n$, $(w_n)_n$ be bounded sequences in $X$ such that $u_n \rightharpoonup u$ weakly in $X$. Then, for every $z \in X$, we have $B_2(v_n w_n, z(u_n - u)) \rightarrow 0$ as $n \rightarrow \infty$.
\end{lemma}
\begin{lemma}\label{lemma:PS}
The functional $\Phi:\mathcal{N}_r \to \R$ satisfies the $(PS)_c$ sequence for all $c>0$.
\end{lemma}
\begin{proof}
Let $(u_n)_n \subset \mathcal{N}_r$ be a $(PS)_c$-sequence. Then, since $u_n \in \mathcal{N}$ 
$$0 = \Phi'(u_n)u_n = \| u_n \|_a^2 + V_0(u_n), $$
and this yields
\begin{equation}\label{eq:ukc}
	c + o(1) = \Phi(u_n) = \frac{\|u_n\|_a^2}{2} + \frac{V_0(u_n)}{4} = \frac{\|u_n\|_a^2}{4} + o(1) = - \frac{V_0(u_n)}{4} + o(1).
\end{equation}
Thus, $(u_n)_n$ is bounded in $H^1(\Omega)$ and $(V_0(u_n))_n$ is a bounded sequence. Further, from the estimate \eqref{eq:B1ineq} we conclude that both sequences $(V_i(u_n))_n$ are bounded. Now, from Corollary \ref{cor:weakconvX} we have that, either $u_n \to 0$ in $L^s(\Omega)$ for $s \in (2,6)$ or, passing to a subsequence, $ u_n \rightharpoonup u \neq 0 $ weakly in $X$. In the first case we see, by the compactness of the embedding $X \hookrightarrow L^{8/3}(\Omega)$, that
$$ 0 = \Phi'(u_n)u_n = \|u_n\|_{a}^2 + V_0(u_n) = \| u_n\|_X^2 + V_2(u_n) + o(1) \geq o(1), $$
so, we conclude that $\|u_k\|_X^2 = V_2(u_k) = o(1)$, which contradicts \eqref{eq:ukc}. Thereby, passing to a subsequence $u_n \rightharpoonup u$ weakly in $X$.

Next, we show that
$$ \| \Phi'(u_n) \|_{X^*} \to 0 \quad \text{ as } n \to \infty,$$
here we remark that we $\| \Phi'(u_n) \|_{X^*} \to 0$ is a statement about the full functional $\Phi : X_r \to \R$, whereas \eqref{eq:ps} holds for the restriction $\Phi:\mathcal{N}_r \to \R$.

Indeed, there are $\mu_n \in \R$ such that
$$ \Phi'(u_n)v = \langle \nabla \Phi(u_n),v \rangle_X + \mu_n J'(u_n)v,$$
where $J(u) = \| u\|_X^2 +V_0(u)$. Now, since $(u_n)_n$ is bounded in $X$, we can set $v = u_n$ to obtain
$$ 0 = \Phi'(u_n)u_n = o(1) + \mu_n J'(u_n)u_n = o(1) + \mu_n (2 \| u_n\|_X^2 + 4 V_0(u_n) ) = o(1) - 8c \mu_n.$$
It follows from this that $\mu_n = o(1)$ and this proves the claim. Now, from the weak convergence and \eqref{eq:estimatesGi} we obtain
\begin{align*}
	o(1) &= \Phi'(u_k)(u_k - u) = \int_{\Omega} a(x) u(x)(u_k-u)(x) + \nabla u(x) \cdot \nabla (u_k-u)(x) dx\\
	&\qquad + B_1(u_k^2,u_k(u_k-u)) + B_2(u_k^2,u_k(u_k-u)) \\
	&= \|u_k\|_a^2 - \|u\|_a^2 +B_2(u_k^2,u_k (u_k-u)) + o(1) = \|u_k\|_a^2 - \|u\|_a^2 +B_2(u_k^2,(u_k-u)^2) + o(1)\\
	&\geq \|u_k\|_a^2 - \|u\|_a^2 + o(1) \geq  o(1),
\end{align*}
where, in the third line we use Lemma \ref{lemma:weakconvergence} in the form
$$ B_2(u_k^2,u(u_k-u)) = o(1),$$
and in the last inequality we use \eqref{eq:B2_lower_bound} to bound $B_2(u_k^2,(u_k-u)^2)$ from below by $o(1)$.

Then, we see that $u_k \to u$ in $H^1(\Omega)$ and $B_2(u_k^2,(u_k-u)^2) \to 0$. The latter implies from Lemma \ref{lemma:B2lemma} that $u_k \to u$ in $X$ as well and this concludes the proof.
\end{proof}

Now we prove our main multiplicity result

\begin{proof}[Proof of Theorem \ref{thm:high_energy_radial}]
We define the sublevel and critical sets
$$\Phi^c := \{ u \in \mathcal{N}_r : \Phi(u) \leq c\} \qquad K_c := \{ u \in \mathcal{N}_r : \Phi(u) = c, \, \Phi'(u) = 0\}.$$

Now, we consider the following \textit{Ljusternik-Schnirelmann values} for the restriction $\Phi : \mathcal{N}_r \to \R$
$$ c_k := \inf\{ c > 0 : \gamma(\Phi^c) \geq k\},$$
where $\gamma$ denotes the \textit{Krasnoselskii genus}, see \cite{s}. Since $\Phi$ is bounded from below over $\mathcal{N}$, see Lemma \ref{lemma:Nehariproperties}, we conclude that $c_k \geq 0$

Now, we consider the family of sets
$$ \Sigma_k := \{ Z \subset \mathcal{N} : Z \text{ is symmetrical, compact and } \gamma(Z) \geq k\} $$
To see that this set is nonempty,  we can fix $r > 0$ small enough that $K(x,y) < 0$ for all $|x-y|<r$; this is made possible by the asymptotic description in Theorem \ref{thm:asymtoticsGreen}. Then we find a family $\{\varphi_j \}_{j = 1}^{k+1} \subset C^{\infty}_c(B_r(0)) \cap X_0$ of functions with pairwise disjoint supports. Observe that $V_0(\sum_j s_j \varphi_j) < 0$ for all convex combinations $\sum_j |s_j| = 1$, since any one of these is still compactly supported in $B_r(0)$. Now, define
$$S := \left\{ \sum_j s_j \varphi_j : \sum_j |s_j| = 1 \right\}.$$
Since this set is isomorphic to $\mathbb{S}^{k}$ we have $\gamma(S) = k$. Now define
$$ Z:= \left\{  t_u u \in \mathcal{N} : u \in S \right\},$$
where $t_u > 0$ is as in Lemma \ref{lemma:Nehariproperties}. It follows that $\gamma(Z) \geq \gamma(S) = k$, thereby $Z \in \Sigma_k$, and, since the set $Z \in \Sigma_k$ constructed above is compact, we see that 
$$ c_k \leq \sup_{u \in Z} \Phi(u)  < \infty.$$
Then, since $\Phi$ satisfies the $(PS)_c$-property for all $c>0$, arguing as in \cite[Chapter II, Theorem 4.2]{struwe}, see also \cite[Proposition 2.10]{ccw}, we see that $c_k$ is a critical point for all $k \in \mathbb{N}$.

Further, since $\Sigma_k \neq \emptyset$ for all $k \in \mathbb{N}$, we conclude that
$$ \lim_{c \to \infty} \gamma(\Phi^c) = \infty,$$
and following a standard argument we see that $c_k \to \infty$. For the ease of the reader we present a brief version here. Suppose by contradiction that $c_k \to c_\infty < \infty$ as $k \to \infty$. From Lemma \ref{lemma:PS}, it follows that $K_{c_\infty}$ is compact and then there is an open neighborhood $K_{c_\infty} \subset N \subset X_r$ such that $\gamma(N) < \infty$. Then, from the deformation lemma for Palais-Smale sequences, see \cite[Chapter II, Theorem 3.11]{struwe}, we obtain $\varepsilon > 0$ and an odd and continuous map $\eta : \mathcal{N}_r \to \mathcal{N}_r$ such that $\eta(\Phi^{c_\infty+\varepsilon} \setminus N) \subset \Phi^{c_\infty-\varepsilon}$ and $\eta(u) = u$ for all $u \notin \Phi^{-1}([c-2\varepsilon,c+2\varepsilon])$. Then, since $c_\infty - \varepsilon < c_k$ for $k$ large enough, we have 
$$\gamma( \eta(\Phi^{c_\infty+\varepsilon} \setminus N) ) \leq \gamma(\Phi^{c_\infty-\varepsilon}) < k,$$
and from the subadditivity and monotonicity of the genus, we conclude that
$$ \gamma(\Phi^{c_\infty + \varepsilon}) \leq  \gamma( \Phi^{c_\infty+\varepsilon} \setminus N ) + \gamma(N) \leq \gamma( \eta(\Phi^{c_\infty+\varepsilon} \setminus N) ) + \gamma(N) < k + \gamma(N) < \infty,$$
which yields a contradiction.

Thus, the values $c_k$ yield the solution pairs $\{\pm u_n\} \subset X_r$. To conclude, we observe from the definition of the Ljusternik-Schnirelmann values that
$$ c_1 = \inf_{\mathcal{N}_r} \Phi, $$
and, since $\mathcal{N}_r$ contains all critical points that are radially symmetric in the $x'$-variable, it follows that $\pm u_1$ is a radial ground state to \eqref{eq:choquard}. The regularity properties of the periodic extensions $\{\pm \overline{u}_n\}_n$ follow from Theorem \ref{thm:solsystem}.
\end{proof}

\section{Fully nontrivial solutions}\label{constant_potentials}
In this section, we focus on the structure arising when $a$ is constant along the $x_3$-variable. We prove that radial ground states of \eqref{eq:periodic_choquard} must be fully nontrivial for large $\ell$. Furthermore, we prove an analogue version of Theorem \ref{thm:high_energy_radial}, which produces fully nontrivial solutions for all $\ell > 0$ in this particular case.

From this point on we will assume at all times that $\partial_3 a(x) \equiv 0$, $\einf_{\Omega} a > 0$ and
$$ a(Ax',x_3) = a(x',x_3) \quad \text{a.e. in } \Omega \text{ for all } A \in O(2). $$
To simplify the notation whenever necessary, we write $f(x',x_3) = f(x')$, for a function $f:\R^3 \to \R$ which does not depend on the $x_3$-variable. 

In this case, we obtain solutions to \eqref{eq:SPsystem} from the planar Choquard equation
\begin{align}\label{eq:Chocplanar}
	-\Delta u + a(x') u + (\log|\cdot| \ast u^2)u = 0 \quad x' \in \R^2,
\end{align} 
by extending trivially to $\R^3$. The equation \eqref{eq:Chocplanar} has a variational structure similar to the periodical case \eqref{eq:choquard}. We only recall here the definition of energy functional associated to it
$$ \Psi(u) := \frac{1}{2} \int_{\R^2} a(x') u(x')^2 + |\nabla u(x')|^2 dx' + \frac{1}{4} \int_{\R^2}\int_{\R^2} \log|x'-y'|u^2(x')u^2(y')dy'dx', $$
defined in the space
$$ B := \left\{ u \in H^1(\R^2) : \int_{\R^2} \log(1+|x|)u^2(x) dx < \infty \right\}. $$
We consider the radial ground state energy level
$$\kappa:=\inf\{\Psi(u): B \setminus \{0\}, u \text{ is radially symmetric and } \Psi'(u)=0\} >0, $$
see \cite[Theorem 1.2]{ccw} for more details on this topic.

Then, we prove the following
\begin{lemma}\label{lemma:solconstant}
	Let $u \in X$ be a weak solution to \eqref{eq:choquard} such that $\partial_3 u =0$ in $\Omega$. Then the function $u(x',x_3)= u(x')$ is a solution to \eqref{eq:Chocplanar}.
\end{lemma}
\begin{proof}
Let $\varphi \in C^{\infty}_c(\R^2)$ and consider its trivial extension to $\R^3$. Then, we have $\varphi \in X$ and thus, from Theorem \ref{thm:green2d}
	\begin{align*}
		0 &=  \int_\Omega a(x)u(x) \varphi(x) + \nabla u(x) \cdot \nabla \varphi(x) dx + \int_{\Omega} \int_{\Omega} K(x,y)u^2(y)u(x)\varphi(x) dx dy,\\
		&= 2\ell \int_{\R^2}a(x')\tilde{u}(x')\tilde{\varphi}(x') + \nabla \tilde{u}(x') \cdot \tilde{\varphi}(x')dx'  + \int_{\Omega} \int_{\Omega} \frac{\log|x'-y'|\tilde{u}^2(y')}{2\ell} u(x')\varphi(x') dy' dx'\\
		&= 2\ell \int_{\R^2}a(x')\tilde{u}(x')\tilde{\varphi}(x') + \nabla \tilde{u}(x') \cdot \tilde{\varphi}(x')dx'  +2\ell \int_{\R^2}  \int_{\R^2} \log|x'-y'|  \tilde{u}^2(y')\tilde{u}(x')\tilde{\varphi}(x') dx' dy'
	\end{align*}
	Then, it follows that $u(x',x_3) = u(x')$ is a weak solution to the equation \eqref{eq:Chocplanar}.
\end{proof}
To make the dependence on $\ell$ explicit in the following proof, we write $\Omega_\ell := \R^2 \times (-\ell,\ell)$, and $K(\cdot,\cdot;\ell)$ for the Green function of \eqref{eq:poisson_per} on $\Omega_\ell$. For the energy functional, we set
$$ \Phi_\ell (u) = \frac{1}{2} \int_{\Omega_\ell} a(x) u^2 + |\nabla u|^2 + \frac{1}{4} \int_{\Omega_\ell} \int_{\Omega_\ell} K(x,y;\ell)u^2(x)u^2(y)dxdy, $$
and we set $\mathcal{N}_\ell$ for the Nehari manifold associated to $\Phi_\ell$.
\begin{proof}[Proof of Theorem \ref{thm:fully_nontrivial}]
	Let $\varphi \in C^{\infty}_c(\R^3) \cap X_r$ be fixed and $\ell_* >0$ large enough that $\text{supp}(\varphi) \subset \Omega_\ell$ for $\ell > \ell_*$. Then, applying Theorem \ref{thm:Green3dim} we have 
	\begin{align*}
		\int_{\Omega_\ell} \int_{\Omega_\ell} K(x,y;\ell)\varphi^2(x)\varphi^2(y)dxdy \to  -\frac{1}{4 \pi} \int_{\R^3} \int_{\R^3}  \frac{\varphi^2(x) \varphi^2(y)}{|x-y|}dx dy < 0
	\end{align*}
	Now, recall from Lemma \ref{lemma:Nehariproperties} that $t_\ell \varphi \in \mathcal{N}_\ell$ if and only if
	$$ t_\ell^{-2} = -\frac{1}{\|\varphi\|_a^2} \int_{\Omega_a} \int_{\Omega_\ell} K(x,y;\ell)u^2(x)u^2(y)dxdy.$$
	So we have
	$$ t_\ell^{-2} \to \frac{1}{4 \pi \|\varphi\|_a^2} \int_{\R^3} \int_{\R^3}  \frac{\varphi^2(x) \varphi^2(y)}{|x-y|}dx dy =: \frac{\mu}{4 \pi \| \varphi\|_a^2} >0 $$
	From this we obtain
	\begin{align*}
		\Phi_\ell(t_\ell \varphi) =  \frac{1}{4} \|t_\ell \varphi\|_a^2 \to \pi \frac{\|\varphi\|_a^4}{\mu} >0,
	\end{align*}
	as $\ell \to \infty$. Then we may find $\ell_*>0$ large enough that
	\begin{equation}\label{eq:upper_bound}
		\inf_{\mathcal{N}_\ell \cap X_r} \Phi_\ell \leq \Phi_\ell (t_\ell \varphi) \leq \frac{2 \pi \|\varphi\|_a^4}{\mu} \quad \text{for } \ell > \ell_*. 
	\end{equation}
	Now, for $\ell > \ell_*$, we let $u_g \neq 0$ be a radial ground state, whose existence is guaranteed by Theorem \ref{thm:high_energy_radial}. If $u_g$ satisfies $\partial_3 u_g = 0$, then, by Lemma \ref{lemma:solconstant} then $u_g$ yields a solution to \eqref{eq:Chocplanar} which implies
	\begin{align*}
		\Phi_\ell(u_g) &= \frac{1}{2} \int_{\Omega_\ell} a(x) u_g^2(x) + |\nabla u_g(x)|^2 \, dx + \frac{1}{4} \int_{\Omega_\ell} \int_{\Omega_\ell} K(x,y;\ell)u_g^2(x)u_g^2(y)dxdy \\
		&\quad = \frac{2\ell}{2} \int_{\R^2} a(x') u_g(x')^2 + |\nabla u_g(x')|^2 dx' + \frac{2\ell}{4} \int_{\R^2}\int_{\R^2} \log|x'-y'|u_g^2(y')u_g(x')dy'dx'\\
		&\quad= 2\ell \Psi(u_g) \geq 2\ell \kappa.
	\end{align*}
	This yields a contradiction with \eqref{eq:upper_bound} by choosing $\ell_*$ large enough that $\ell_* \kappa > \pi \|\varphi\|_a^4/\mu $. Then $\partial_3 u_g \neq 0$ for $\ell > \ell_*$ as claimed.
\end{proof}

To close this section, we observe that, in this case, out of all the solution pairs provided by Theorem \ref{thm:high_energy_radial}, we can only guarantee that $\pm u_1$ is a fully nontrivial solution, as we proved above.

This can be remedied by slightly changing our setting. Consider the isometric transformation $\inv : H^1_p(\Omega) \to H^1_p(\Omega)$ given by
\begin{equation}\label{eq:action_inv}
	(\inv \ast u)(x) = -T_{(0,\ell)} u (x)=-\overline{u}(x',x_3-\ell)  \qquad x \in \Omega,
\end{equation}
where $T$ is the translation operator \eqref{eq:translation_op}. Observe that, since $u \in H^1_p(\Omega)$ then $\inv^2 = \id_X$. Then let $G$ be the group generated by $O(2)$ and $\inv$ which acts on $X$ according to \eqref{eq:action_O2} and \eqref{eq:action_inv} and consider the subspace of $G$-invariant functions
$$ X_G := \{ u \in X : g \ast u = u \quad \text{for all } g \in G\}.$$
We conclude with the following
\begin{proof}[Proof of Theorem \ref{thm:high_energy_nontrivial}]
	Observe that if $\partial_3 a \equiv 0$ then the functional $\Phi$ is $G$-invariant and a critical point $u \in X_G$ of the functional $\Phi: X \to \R$ is a fully nontrivial solution to \eqref{eq:choquard}, since
	$\overline{u}(x',x_3- \ell) = -u(x',x_3)$ a.e. in $\Omega$. Moreover, the analogous version of the principle of symmetric criticality, Lemma \ref{lemma:symmetriccriticality}, holds for the $G$-action. Then, we reproduce the proofs in Section \ref{multiplicity} to obtain the solutions pairs $\{\pm u_n\}_n$ in the $G$-invariant space $X_G$, where $\Phi(\pm u_1) = c_G$. The regularity of the extensions $\{ \pm \overline{u}_n\}_n$ is obtained from Theorem \ref{thm:solsystem}.
\end{proof}

\textbf{Acknowledgments}. This research was done with the full support of DAAD-\textit{Deutscher Akademischer Austauschdienst} (Germany) and the author wishes to express his gratitude with this institution. He would also like to thank Prof. Dr. Tobias Weth for the helpful discussions that made this work possible.

\medskip

\begin{flushleft}
	\textbf{Omar Cabrera Chavez}\\
	Institut für Mathematik\\
	Goethe Universität Frankfurt am Main\\
	D-60054 Frankfurt am Main, Germany\\
	\texttt{cabrera@math.uni-frankfurt.de} 
\end{flushleft}

\end{document}